\documentclass[a4paper,11pt]{article}
\usepackage[ansinew]{inputenc}
\usepackage[spanish,english]{babel}
\usepackage{amsmath,amsfonts,amssymb,amsthm}
\usepackage{graphicx}
\usepackage{subfig}
\usepackage[usenames,dvipsnames]{color}
\usepackage{lineno}
\usepackage{enumerate}
\usepackage[colorlinks=true,citecolor=black,linkcolor=black,urlcolor=blue]{hyperref}
\usepackage{amscd,latexsym}

\theoremstyle{plain}
\newtheorem{theorem}{Theorem}
\newtheorem{lemma}{Lemma}

\newtheorem{claim}{Claim}
\newtheorem{conj}{Conjecture}

\newtheorem*{invariant*}{Invariant}

\newcommand\blfootnote[1]{%
  \begingroup
  \renewcommand\thefootnote{}\footnote{#1}%
  \addtocounter{footnote}{-1}%
  \endgroup
}

\date{}

\title{Total domination in plane triangulations
}

\author{M. Claverol\thanks{{\tt merce.claverol@upc.edu}. Universitat Polit\`{e}cnica de Catalunya, Spain.}
\and A. Garc\'\i a\thanks{{\tt olaverri@unizar.es}. IUMA, Universidad de Zaragoza, Spain.}
\and G. Hern\'andez\thanks{{\tt gregorio@fi.upm.es}. Universidad Polit\'ecnica de Madrid, Spain.}
\and C. Hernando \thanks{{\tt carmen.hernando@upc.edu}. Universitat Polit\`{e}cnica de Catalunya, Spain.}
\and M. Maureso\thanks{{\tt montserrat.maureso@upc.edu}. Universitat Polit\`{e}cnica de Catalunya, Spain.}
\and M. Mora \thanks{{\tt merce.mora@upc.edu}. Universitat Polit\`{e}cnica de Catalunya, Spain.}
\and J. Tejel\thanks{{\tt jtejel@unizar.es}. IUMA, Universidad de Zaragoza, Spain.}}

\begin{document}

\maketitle

\blfootnote{\begin{minipage}[l]{0.3\textwidth} \includegraphics[trim=10cm 6cm 10cm 5cm,clip,scale=0.15]{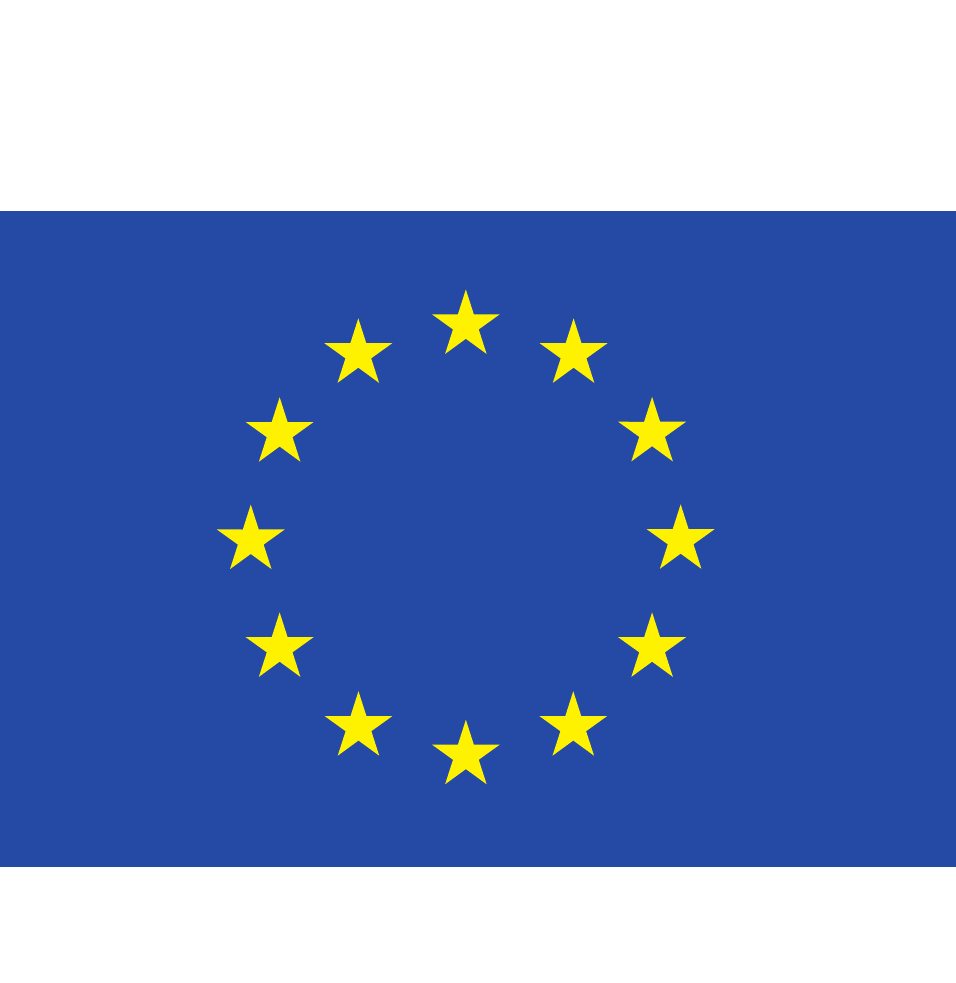} \end{minipage}  \hspace{-2cm} \begin{minipage}[l][1cm]{0.7\textwidth}
 	  This project has received funding from the European Union's Horizon 2020 research and innovation programme under the Marie Sk\l{}odowska-Curie grant agreement No 734922.
 	\end{minipage}}

\begin{abstract}
A total dominating set of a graph $G=(V,E)$ is a subset $D$ of $V$ such that every vertex in $V$ is adjacent to at least one vertex in $D$. The total domination number of $G$, denoted by $\gamma _t (G)$, is the minimum cardinality of a total dominating set of $G$. A near-triangulation is a biconnected planar graph that admits a plane embedding such that all of its faces are triangles except possibly the outer face. We show in this paper that $\gamma _t (G) \le \lfloor \frac{2n}{5}\rfloor$ for any near-triangulation $G$ of order $n\ge 5$, with two  exceptions.
\end{abstract}

\section{Introduction}

Let $G=(V,E)$ be a simple graph. A {\em dominating set} of $G$ is a subset $D\subseteq V$ such that every vertex not in $D$ is adjacent to at least one vertex in $D$. The {\em domination number} of $G$, denoted by $\gamma (G)$, is defined as the minimum cardinality of a dominating set of $G$. Total dominating sets are defined in a similar way.
A subset $D\subseteq V$ such that every vertex in $V$ (including the vertices in $D$) is adjacent to a vertex in $D$ is called a {\em total dominating set} ({\em TDS} for short) of $G$. The {\em total domination number}, denoted by $\gamma _t (G)$, is the minimum cardinality of a total dominating set of $G$. Since a total dominating set of a graph $G$ is also a dominating set of $G$, the following inequality trivially holds $\gamma (G)\le \gamma _t (G)\le 2 \gamma (G)$.

Domination and total domination in graphs have been widely studied in the literature. We refer the reader to~\cite{Haynes98, Haynes98bis, Henning13} for excellent books on these topics and to~\cite{Henning09} for a survey on total domination.

\begin{figure}[htb]
\centering
\includegraphics[scale=0.56,page=1]{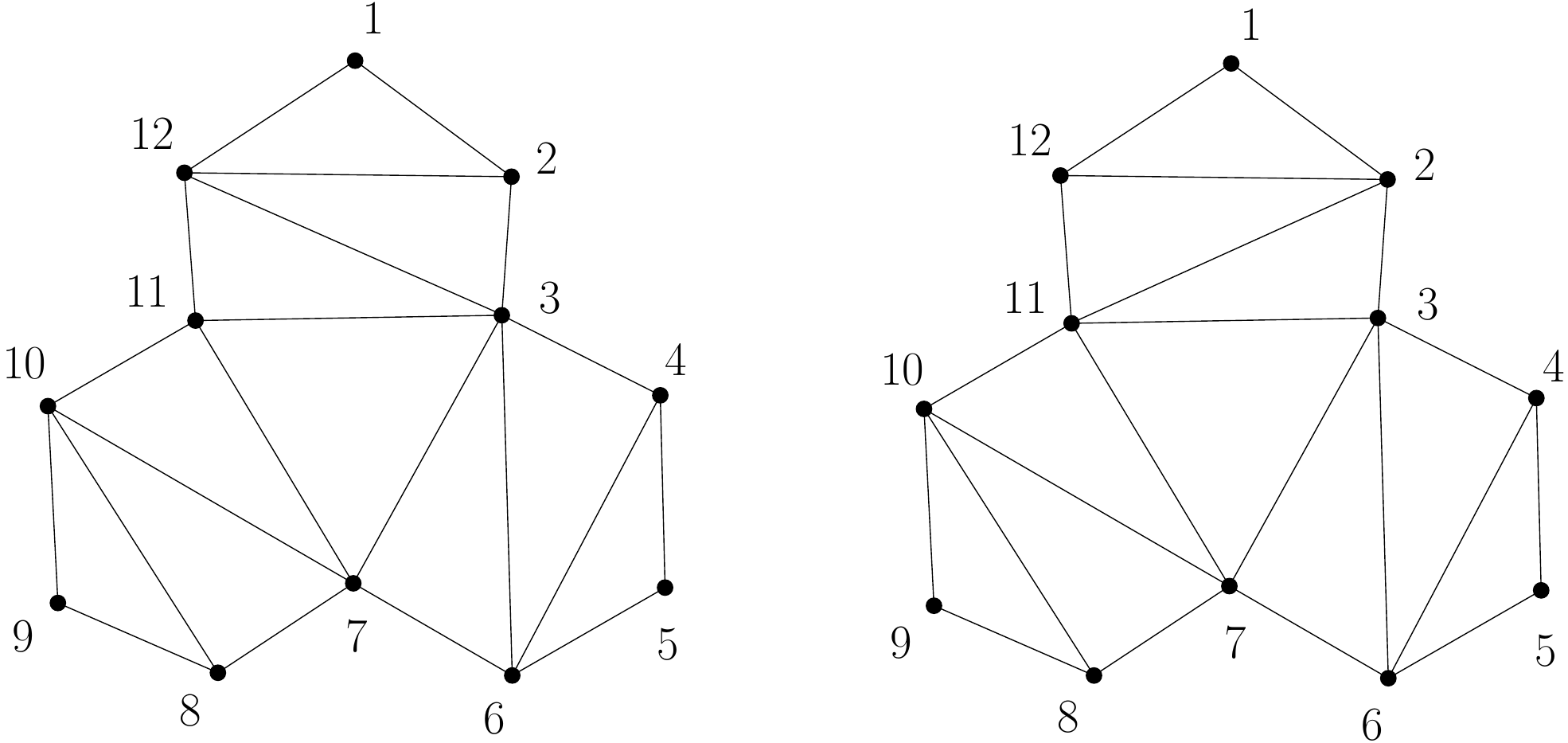}%\hspace{2mm}
\caption{The two 12-vertex graphs $H_1$ and $H_2$}\label{fig:ExceptionCases}
\end{figure}

Given a graph $G$, it is well-known that computing $\gamma (G)$ or $\gamma_t(G)$ is an NP-hard problem, even when restricted to planar graphs.
%~\cite{Garey79, Henning09}.
Hence, studying lower or upper bounds on the (total) domination number in some classes of graphs has been of interest during the last few years. In particular, for planar graphs, Matheson and Tarjan proved in~\cite{Matheson96} that $\gamma (G) \le \lfloor \frac{n}{3}\rfloor$ for any $n$-vertex triangulated disc $G$. In the literature, triangulated discs are also called near-triangulations. A {\em near-triangulation} is a biconnected planar graph that has a plane embedding such that all of its faces are triangles except possibly the outer face.  When the outer face is also a triangle, a near-triangulation is a {\em triangulation} or maximal planar graph. Note that a near-triangulation can be obtained by removing one vertex from a triangulation.

In the same paper~\cite{Matheson96}, it is also conjectured that $\gamma (G) \le \lfloor \frac{n}{4}\rfloor$ for any $n$-vertex triangulation $G$. King and Pelsmajer proved this conjecture in~\cite{King10} for triangulations of maximum degree $6$, and Plummer et al. proved in~\cite{Plummer16} that if $G$ is an $n$-vertex Hamiltonian triangulation with minimum degree at least 4, then $\gamma (G) \le \max \{\lceil 2n/7\rceil, \lfloor 5n/16\rfloor\}$. The upper bound $\frac{n}{3}$ for triangulations has been recently improved by \v{S}pacapan~\cite{Spacapan18}, showing that $\gamma (G) \le \frac{17}{53}n$ for any $n$-vertex triangulation $G$.

Maximal outerplanar graphs are a special class of near-triangulations. A {\em maximal outerplanar graph}, MOP for short, is a near-triangulation such that all of its vertices belong to the boundary of the outer face.
MOPs have additional properties that allow one to improve (or to prove) bounds for different types of problems on graphs. In~\cite{Matheson96}, in addition to proving that $\gamma (G) \le \frac{n}{3}$ for any $n$-vertex planar graph $G$, it is proved that this upper bound is tight for MOPs. In fact, the upper bound $\frac{n}{3}$ on the domination number
in MOPs was already implicitly proved by Fisk~\cite{Fisk78}.
In~\cite{Campos13, Tokunaga13}, it is shown  that $\gamma (G) \le (n+k)/4$, where $k$ is the number of vertices of degree 2 in a MOP $G$.
Dorfling et al. proved in~\cite{Dorfling17} that $\gamma _t (G) \le \frac {n+k}{3}$ for a MOP $G$ of order $n$ with $k$ vertices of degree 2. The same authors proved in~\cite{Dorfling16} that apart from the graphs $H_1$ and $H_2$ shown in Figure~\ref{fig:ExceptionCases}, $\gamma _t (G) \le \lfloor \frac {2n}{5} \rfloor$ for a MOP $G$ of order $n\ge 5$. In~\cite{Lemanska17}, Lemanska et al. presented an alternative proof of this last result. The reader is referred to~\cite{Canales16, Canales18, Claverol19, Henning19, Lemanska19} for other results in MOPs related to some variants of the domination concept.

In this paper, we extend the result proved in~\cite{Dorfling16, Lemanska17} to the family of near-triangulations and we show that $\gamma_t(G) \le \lfloor \frac {2n}{5} \rfloor$ for any near-triangulation $G$ of order $n\ge 5$, apart from the graphs $H_1$ and $H_2$. Thus, we improve the best known upper bound $\frac{6}{11}n$ on the total domination number of $n$-vertex near-triangulations. This last bound follows from the fact that a near-triangulation is 2-connected and from the following result proved in~\cite{Henning09bis}: If $G$ is a 2-connected graph of order $n>18$, then $\gamma_t (G) \le \frac{6}{11}n$.

The upper bound $\lfloor \frac {2n}{5} \rfloor$ on the total domination number in near-triangulations is proved in Section~\ref{sec:bound}. The proof is based on induction and combines common techniques used when proving results for MOPs, as the ones described in~\cite{Lemanska17}, with techniques related to what we call {\em reducible} and {\em irreducible} near-triangulations, and {\em terminal polygons} in irreducible near-triangulations. These concepts are defined in Section~\ref{sec:bound}. In the induction process, the two exception graphs $H_1$ and $H_2$ can appear after removing some vertices or some edges from a near-triangulation. For these two graphs, induction cannot be applied since their total domination numbers are greater than $\lfloor \frac {2n}{5} \rfloor$. For this reason, we explain in Section~\ref{sec:dominating} how to obtain suitable total dominating sets for some graphs involving $H_1$ and $H_2$ that will be used in the inductive proof.
Section~\ref{sec:near} is devoted to review some known properties for near-triangulations, and to show some special cases in which the removal of some vertices or the contraction of some edges from a near-triangulation, results in another near-triangulation. These cases will be needed in the inductive proof. We conclude the paper with some remarks in Section~\ref{sec:con}.

\section{Near-triangulations and some of their properties}\label{sec:near}

For the sake of simplicity, throughout the paper the term near-triangulation will refer to a near-triangulation $T=(V,E)$ that has been drawn in the plane without crossings, using straight-line segments, such that all of its faces are triangles except possibly the outer face (see Figure~\ref{fig:NearTriangulationa}). Such a drawing always exists by F\'ary's Theorem~\cite{Fary48}. We assume that the boundary of the outer face is given by the cycle $C=(u_1,u_2,\ldots ,u_h,u_1)$, with its $h\ge 3$ vertices in clockwise order. In this way, we can refer to boundary edges and vertices (the edges and vertices of $C$), interior vertices (the vertices not in $C$), and diagonals (edges connecting two non-consecutive vertices of $C$). Recall that if $h=3$, then $T$ is a triangulation and if $h=|V|$, then $T$ is a MOP.

In~\cite{Lemanska17}, the authors use induction to prove that $\gamma_t(T) \le \lfloor \frac {2n}{5} \rfloor$ for a MOP $T$ of order $n\ge 21$. The two main properties they use are that after contracting a boundary edge of $T$, the resulting graph is again a MOP, and that there is always a diagonal dividing $T$ into two MOPs, leaving $5,6,7$ or $8$ consecutive boundary edges of $C$ in the smallest one. However, these two properties are not true for arbitrary near-triangulations. Sometimes, there are no diagonals dividing a near-triangulation $T$ into smaller near-triangulations, and even in the case that such diagonals exist, a diagonal leaving $5,6,7$ or $8$ consecutive boundary edges in the smallest near-triangulation cannot be chosen. Besides, in general, after contracting a boundary edge, the resulting graph is not a near-triangulation. Therefore, we cannot follow in our inductive proof the same steps as described in~\cite{Lemanska17}, although we will use some of the ideas given in that paper.

\begin{figure}[ht]
	\centering
	\subfloat[]{
		\includegraphics[scale=0.56,page=2]{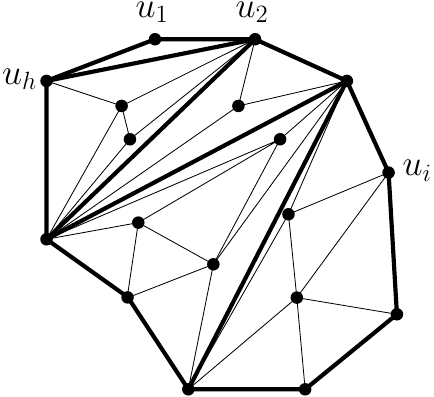}
		\label{fig:NearTriangulationa}
	}~
	\subfloat[]{
		\includegraphics[scale=0.56,page=3]{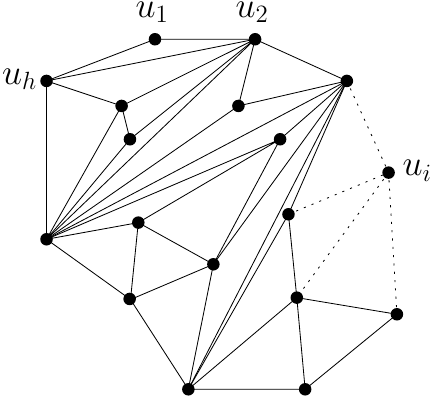}
		\label{fig:NearTriangulationb}
	}~
	\subfloat[]{
		\includegraphics[scale=0.56,page=4]{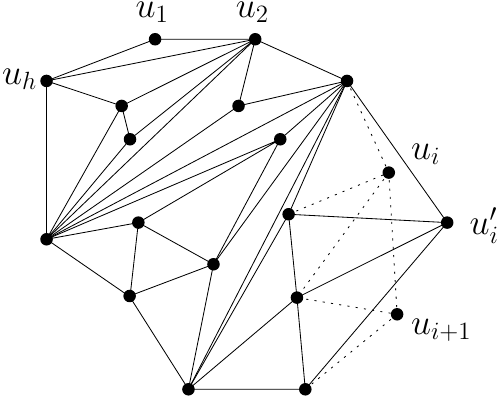}
		\label{fig:NearTriangulationc}
	}
	\caption{
		(a) A near-triangulation. The thick segments correspond to $T[C]$.
		(b) Removing a vertex of degree 2 in $T[C]$.
        (c) Contracting the edge $(u_i,u_{i+1})$ to the vertex $u'_i$.
	}
	\label{fig:NearTriangulations}
\end{figure}

We show in this section several cases in which the removal of some vertices or the contraction of some edges from a near-triangulation results in another near-triangulation. These cases will be enough for our purposes. Before stating them, we shall give some terminology and some properties.

Given a near-triangulation $T$ with boundary cycle $C=(u_1,u_2,\ldots ,u_h,u_1)$, we use $T[C]$ to denote the subgraph of $T$ induced by the vertices in $C$ (see Figure~\ref{fig:NearTriangulationa}). Observe that $T[C]$ is always Hamiltonian and outerplane (all the vertices belong to the boundary of the outer face). The following result for a Hamiltonian outerplanar graph is well-known.

\begin{lemma}\label{lem:grado2}
Let $G$ be a Hamiltonian outerplanar graph of order $n\ge 4$. Then, $G$ contains at least two non-adjacent vertices of degree 2.
\end{lemma}
Given a graph $G=(V,E)$, the graph obtained from $G$ by deleting the vertices $\{v_1, \ldots , v_k\}$ and all their incident edges is denoted by $G-\{v_1, \ldots , v_k\}$.
%Using Euler's formula,
It is straightforward to prove the following lemma for near-triangulations (see Figure~\ref{fig:NearTriangulationb}).

\begin{lemma}\label{lem:deletion}
Let $T$ be a near-triangulation of order $n\ge 4$ with boundary cycle $C$. Then, $T-\{ v\} $ is a near-triangulation if and only if $v$ is an interior vertex of degree 3 or $v$ is a vertex of degree 2 in $T[C]$.
\end{lemma}

Let $G=(V,E)$ be a graph and let $e=(v_i,v_j)$ be an edge of $G$. We use $G - e$ to denote the graph obtained from $G$ by removing $e$, and $G / e$ to denote the graph obtained from $G$ by contracting the edge $e$, that is, the simple graph obtained from $G$ by deleting $v_i, v_j$ and all their incident edges, adding a new vertex $w$ and connecting $w$ to each vertex $v$ that is adjacent to either $v_i$ or $v_j$ in $G$ (see Figure~\ref{fig:NearTriangulationc}). Observe that by Euler's formula, contracting an edge $e=(v_i,v_j)$ from a triangulation $T$ results in another triangulation if and only if $v_i$ and $v_j$ have exactly two common neighbors. Besides, the two endpoints of an edge $e=(v_i,v_j)$ of $T$ have exactly two common neighbors if and only if the edge $e$ is not an edge of a {\em separating triangle} (a triangle containing vertices inside and outside).

We say that an edge $e$ of a near-triangulation $T$ is {\em contractible} if the graph $T / e$ is also a near-triangulation. Since by adding a vertex $w$ in the outer face of $T$ and by connecting $w$ to the vertices in $C$ (the boundary cycle associated with $T$) we obtain a triangulation, then we have the following lemma.

\begin{lemma}\label{lem:contraction}
Let $T$ be a near-triangulation with boundary cycle $C$ and let $e$ be an edge of $T$. Then, the edge $e$ is contractible if and only if $e$ is neither a diagonal of $T$ nor an edge of a separating triangle of $T$.
\end{lemma}

The following lemma summarizes some of the cases in which we obtain new near-triangulations after removing vertices from a near-triangulation.

\begin{lemma}\label{lem:RemovingPoints}
Let $T$ be a near-triangulation with boundary cycle $C=(u_1,\ldots ,u_h,u_1)$. Suppose that $T$ contains at least one interior vertex and has no diagonals. Let $u_i$ be a vertex in $C$. Then:
\begin{itemize}
  \item[i)] $T-\{ u_i\} $ is also a near-triangulation.
  \item[ii)] Assuming that $T$ contains at least two interior vertices, there exists a vertex $u_j$ with $i\le j<i-1+h$ (mod $h$) and an interior vertex $v_j$ adjacent to $u_j$ such that $T-\{ u_i,u_{i+1},\ldots , u_j,v_j\} $ is a near-triangulation.       In addition, the edge $(u_j,v_j)$ is contractible in $T$.
  \item[iii)] If the edge $e_i=(u_{i-1},u_i)$ is not contractible in $T$, then there exists an interior vertex $v_i$ adjacent to $u_i$ such that $T-\{ u_i,v_i\}$ is a near-triangulation.
\end{itemize}
\end{lemma}

\begin{proof}
Since the starting vertex of $C$ is arbitrary, we may assume without loss of generality that $u_i$ is $u_2$.

i) There are no diagonals in $T$, so the degree of $u_2$ in $T[C]$ is 2. Thus, the statement follows from Lemma~\ref{lem:deletion}.

ii) Let $u_{1},w_1,\ldots ,w_k,u_{3}$ be the set of neighbors of $u_2$ in $T$, in counterclockwise order. Since there are no diagonals in $T$, $u_2$ is a vertex of degree 2 in $T[C]$ and $k\ge 1$. By Lemma~\ref{lem:deletion}, after removing $u_2$ we obtain a new near-triangulation $T_2=T-\{ u_2\} $ with boundary cycle $C_2=(u_1,w_1,\ldots ,w_k,u_{3},u_4,\ldots ,u_1)$ (see Figure~\ref{fig:lemma3b}).

We repeat this operation and we remove from $T_2$ the first vertex $w$ of degree 2 in $T_2[C_2]$, clockwise from $u_1$. By Lemma~\ref{lem:deletion}, we obtain again a near-triangulation $T_3=T-\{ u_2, w\} $ with boundary cycle $C_3$. Iterating this process, we obtain a sequence of near-triangulations $T_2, T_3, \ldots , T_j, T_{j+1}$, where $T_{i+1}$ is obtained from $T_{i}$, for $i=2, \ldots ,j$, by removing from $T_i$ the first vertex $w$ of degree 2 in $T_i[C_i]$, clockwise from $u_1$, and where we have stopped the process the first time that $w$ is an interior vertex in $T$. Hence, $T_{j+1} = T-\{ u_2,u_{3},\ldots , u_j,v_j\} $, for some interior vertex $v_j$. See Figure~\ref{fig:lemma3} for an illustration of this process. Next we prove the following claim.

\begin{figure}[ht]
	\centering
	\subfloat[]{
		\includegraphics[scale=0.56,page=5]{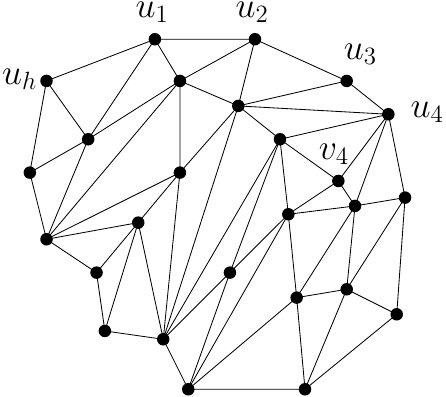}
		\label{fig:lemma3a}
	}~~
	\subfloat[]{
		\includegraphics[scale=0.56,page=6]{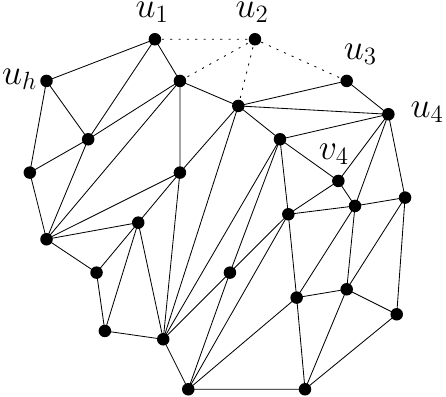}
		\label{fig:lemma3b}
	}~~
	\subfloat[]{
		\includegraphics[scale=0.56,page=7]{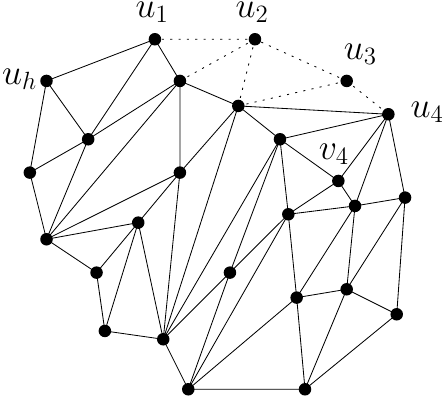}
		\label{fig:lemma3c}
%	}~~~~
	}\\
	\subfloat[]{
		\includegraphics[scale=0.56,page=8]{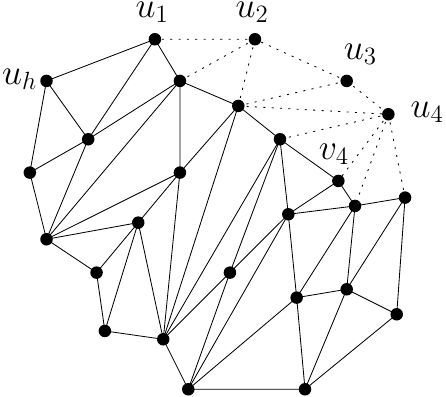}
		\label{fig:lemma3d}
	}~~
	\subfloat[]{
		\includegraphics[scale=0.56,page=9]{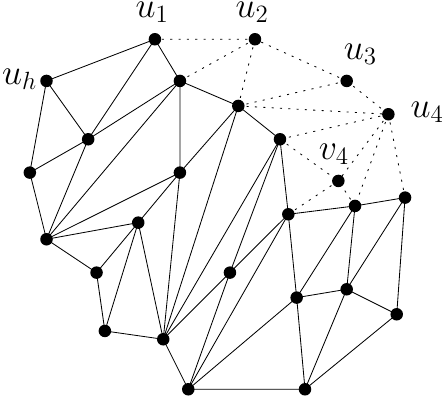}
		\label{fig:lemma3e}
	}
	\caption{
%		Illustrating the proof of Lemma~\ref{lem:RemovingPoints}(ii).
        (a) A near-triangulation $T$ without diagonals. (b), (c), (d) and (e) Obtaining the near-triangulations $T_2, T_3, T_4$ and $T_5$ by removing successively the vertices $u_2, u_3, u_4$ and $v_4$.
	}
	\label{fig:lemma3}
\end{figure}

\begin{claim}\label{claim1}
For $i=2, \ldots , j$, the boundary cycle $C_i$ of $T_i$ consists of the following vertices and in this (clockwise) order: The vertex $u_1$, some vertices that are interior in $T$, and the boundary vertices $u_{i+1}, u_{i+2}, \ldots , u_h$.
\end{claim}
\noindent{\em Proof of Claim~\ref{claim1}}. The proof is by induction. The claim is obviously true for $T_2$, the base case. Assume that the claim is true for $T_2, \ldots , T_{i}$ and that $C_{i}$ is $(u_1, x_1, \ldots , x_l, u_{i+1}, u_{i+2}, \ldots , u_h, u_1)$, with $x_1, \ldots , x_l$ being interior vertices in $T$. Let us prove that $T_{i+1}$ satisfies the claim. By the construction of $T_{i+1}$, none of the vertices $x_1, \ldots , x_l$ has degree 2 in $T_i[C_i]$. Let us see that $u_{i+1}$ is the first vertex of degree 2 in $T_i[C_i]$ from $u_1$.  Assume to the contrary that its degree  in $T_i[C_i]$ is greater than 2, so there is a diagonal $(u_{i+1},y)$ in $T_i[C_i]$. Since $T$ has no diagonals, the vertex $y$ necessarily is one of the vertices in $\{x_1, \ldots , x_{l-1}\}$, say $x_k$. But then, by Lemma~\ref{lem:grado2}, in the subgraph induced by the vertices $x_k, x_{k+1}, \ldots , x_l, u_{i+1}$ there is a vertex $x_j$ of degree 2 different from $x_k$ and $u_{i+1}$, that also is a vertex of degree 2 in $T_i[C_i]$, which is a contradiction. Hence, $u_{i+1}$ is a vertex of degree 2 in $T_i[C_i]$. If we remove it from $T_i$, then the new cycle $C_{i+1}$ corresponding to $T_{i+1}$ is obtained from $C_i$ by adding the neighbors of $u_{i+1}$ in $T_i$ between $x_l$ and $u_{i+2}$. Therefore, the claim follows.
 \hfill $\square$

From the claim, the set of boundary vertices removed to obtain $T_j$ is $\{u_2, \ldots, u_j\}$, as required. Let us now see that during the previous process, there is always a first time in which an interior vertex in $T$ can be removed. Assume that the process does not finish before removing $u_{h-1}$. By removing $u_{h-1}$, we obtain a near-triangulation $T_{h-1}$ with boundary cycle $C_{h-1}=(u_1, x_1, \ldots , x_l, u_h,u_1)$. By hypothesis, $T$ contains at least two interior vertices, so $T_{h-1}$ is not a triangle. Thus, by Lemma~\ref{lem:grado2}, $T_{h-1}[C_{h-1}]$ contains a vertex $x_j$ of degree 2, different from $u_1$ and $u_h$, that must be an interior vertex in $T$, and can be removed from $T_{h-1}$ to obtain a near-triangulation by Lemma~\ref{lem:deletion}.

Let $v_j$ be the interior vertex in $T$ removed from $T_j$ to obtain $T_{j+1}$. To finish this part of the proof, we need to show that $u_j$ and $v_j$ are adjacent and that $(u_j,v_j)$ is contractible in $T$. Let $C_{j-1} = (u_1, y_1, \ldots , y_m, u_{j}, u_{j+1}, \ldots , u_h, u_1)$ be the boundary cycle of $T_{j-1}$. By hypothesis, none of the vertices $y_1, \ldots , y_m,$ has degree 2 in $T_{j-1}[C_{j-1}]$. Since the vertex $u_j$ has degree 2 in $T_{j-1}[C_{j-1}]$, it is not connected to any of $\{y_1, \ldots , y_{m-1}\}$, so when removing $u_j$ from $T_{j-1}$, the only vertex among $\{y_1, \ldots , y_m\}$ that could decrease its degree in $T_j[C_j]$ in relation to its degree in $T_{j-1}[C_{j-1}]$ is precisely $y_m$. Therefore, $v_j$ is either $y_m$ or one of the new vertices that appear in $C_j$. Since all of these vertices are neighbors of $u_j$, then $u_j$ and $v_j$ are adjacent.

Let us prove that $(u_j,v_j)$ is contractible in $T$. Assume to the contrary that the edge $(u_j,v_j)$ is not contractible. $T$ has no diagonals, hence there exists a separating triangle $\Delta=(u_j,v_j,v)$ in $T$ by Lemma~\ref{lem:contraction}. The vertex $u_j$ has degree 2 in $T_{j-1}[C_{j-1}]$ and $T$ has no diagonals, so all the neighbors of $u_j$ in $T$ must belong to $C_j$ except for $u_{j-1}$. The vertex $v_j$ has degree 2 in $T_{j}[C_{j}]$, hence the only neighbors of $u_j$ adjacent to $v_j$ are the predecessor and the successor of $v_j$ in $C_j$. Thus, $v$ must be one of these two vertices. But in both cases, $\Delta$ would be empty, contradicting that $\Delta$ is separating. Therefore, $(u_j,v_j)$ is contractible in $T$.

iii) Suppose that the edge $(u_{1},u_2)$ is not contractible. Since $T$ contains no diagonals, by Lemma~\ref{lem:contraction} this edge must belong to a separating triangle $\Delta=(u_{1},u_2,w)$ containing some vertices inside, with $w$ being an interior vertex in $T$. If $\Delta$ only contains a vertex $w_i$, then $T-\{u_2,w_i\}$ is clearly a near-triangulation by Lemma~\ref{lem:deletion}, because $w_i$ is an interior vertex of degree 3 and $u_2$ is a vertex of degree 2 in $T[C]$. If $\Delta$ contains two or more vertices, then part ii) of this lemma can be applied to the triangulation $T'$ induced by $\Delta$ and its interior vertices, so there is a vertex $w_i$ inside $\Delta$ such that $T'-\{u_2,w_i\}$ is a near-triangulation. As a consequence, $T-\{u_2,w_i\}$ is also a near-triangulation.
\end{proof}

To finish this section, we show that for a boundary vertex, there is always a contractible edge incident to it.

\begin{lemma}\label{lem:ContractibleEdges}
Let $T$ be a near-triangulation of order $n\ge 5$, with boundary cycle $C=(u_1,\ldots ,u_h,u_1)$, and let $u_i$ be a vertex in $C$. Then,

\begin{itemize}
  \item[i)] If $u_i$ has a neighbor not in $C$, then there exists an interior vertex $v$ such that the edge $(u_i,v)$ is contractible. %In this case the contraction $T_n\setminus e$ has one interior vertex less than $T_n$.
  \item[ii)] If all neighbors of $u_i$ are in $C$, then the edges $(u_{i-1},u_i)$ and $(u_i,u_{i+1})$ are contractible. %, and $T_n\setminus e$ has a boundary edge less than $T_n$ except if $h=3$.
\end{itemize}
\end{lemma}
\begin{proof}
i) Suppose that the edge $e=(u_i,v)$ is not contractible, with $v\notin C$. Then $e$ must be an edge of a separating triangle $\Delta=(u_i,v,w)$. All vertices inside $\Delta$  are interior vertices in $T$, and the subgraph induced by $\Delta$ and its interior vertices is a triangulation $T'$. If $\Delta$ contains at least two vertices, then, by Lemma~\ref{lem:RemovingPoints}(ii), there exists an interior vertex $v'$ such that $T'-\{u_i,v'\}$ is a near-triangulation and $(u_i,v')$ is contractible in $T'$. But this edge is also contractible in $T$. If $\Delta$ only contains an interior vertex $z$, then the edge $(u_i,z)$ is clearly contractible in $T$.

ii) Suppose that the edge $e=(u_{i-1},u_i)$ is not contractible. This edge is not a diagonal, hence there exists a separating triangle $\Delta=(u_{i-1},u_i,u)$ containing at least one interior vertex. Thus, at least one of these interior vertices must be adjacent to $u_i$, which is a contradiction because we are assuming that all neighbors of $u_i$ belong to $C$. Therefore,  $(u_{i-1},u_i)$ is contractible. By the same argument, the edge $(u_{i},u_{i+1})$ is also contractible.
\end{proof}

\section{Dominating sets for some near-triangulations}\label{sec:dominating}

In this section we show how to build (total) dominating sets in some special cases of near-triangulations. These dominating sets are needed in the proof of the main theorem. We first give the following results for triangulated pentagons and hexagons, and MOPs in general~\cite{Dorfling16, Lemanska17}.

\begin{lemma}[\cite{Dorfling16, Lemanska17}] \label{lem:pentagon}
Let $T$ be a MOP of order 5 and let $C=(u_1,\ldots ,u_5,u_1)$ be its boundary cycle. For every vertex $u_i$, there exists a TDS in $T$ whose size is 2 and contains $u_i$.
\end{lemma}

\begin{lemma}[\cite{Dorfling16, Lemanska17}]\label{lem:hexagon}
Let $T$ be a MOP of order 6 and let $C=(u_1,\ldots ,u_6,u_1)$ be its boundary cycle. For every pair $u_i,u_{i+1}$ of consecutive vertices in $C$, there exists a TDS in $T$ whose size is 2 and contains either $u_i$ or $u_{i+1}$.
\end{lemma}

\begin{theorem}[\cite{Dorfling16, Lemanska17}]\label{the:MOPs}
If $T$ is a MOP of order $n\ge 5$ and $T\notin \{H_1,H_2\}$, then $\gamma _t (T) \le \lfloor \frac {2n}{5} \rfloor$.
\end{theorem}

The following lemma provides total dominating sets in some cases that involve the graphs $H_1$ and $H_2$.

\begin{lemma}\label{lem:ExceptionCases}
Let $T$ be a near-triangulation with boundary cycle $C=(u_1, \ldots , u_h,u_1)$.
\begin{itemize}
  \item [I)] For every vertex $u_i\in C$, $T$ has a TDS of size 5 containing $u_i$ if one of the following cases holds:
      \begin{itemize}
  \item[i)] $T$ is either $H_1$ or $H_2$.
  \item[ii)] $T-u_i$ is either $H_1$ or $H_2$.
  \item[iii)] $T-\{u_i,v_i\}$ is either $H_1$ or $H_2$ for some interior vertex $v_i$ adjacent to $u_i$.
  \item[iv)] $T/e$ is either $H_1$ or $H_2$ by contracting some edge $e$ incident with $u_i$.  %      with $u_i$ as one of its endpoints.
\end{itemize}
  \item[II)] For every edge $e_i=(u_i,u_{i+1})$ (where $i+1$ is taken modulo $h$), $T$ has a TDS of size $4$ containing $u_i$ or $u_{i+1}$ if $T-e_i$ is $H_1$ or $H_2$.
\end{itemize}
\end{lemma}
\begin{proof}
We prove the lemma assuming that $H_1$ is $T$ or the graph obtained from $T$. The analysis is totally analogous if $H_2$ is $T$ or the graph obtained from $T$. Let $\Delta$ be the central triangle of $H_1$, consisting of the vertices $w_1, w_2$ and $w_3$. See Figure~\ref{fig:lemma7a}. The three triangles that contain the three vertices of degree 2 are denoted by $\Delta_1,\Delta_2$ and $\Delta_3$, respectively, where $w_i$ is not adjacent to any vertex in $\Delta_i$, for $i=1,2,3$.

\begin{figure}[ht]
	\centering
	\subfloat[]{
		\includegraphics[scale=0.54,page=10]{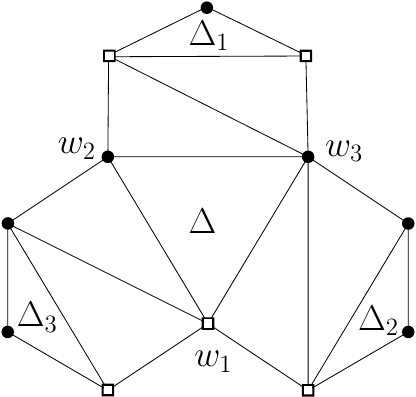}
		\label{fig:lemma7a}
	}~~	
    \subfloat[]{
		\includegraphics[scale=0.54,page=11]{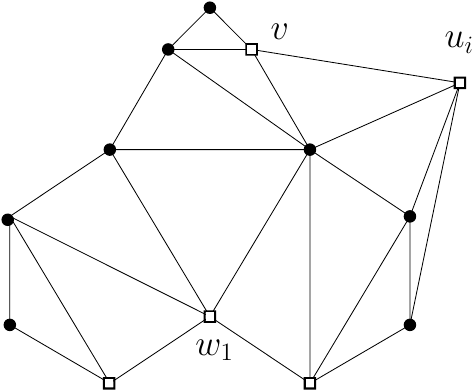}
		\label{fig:lemma7b}
	}~~
	\subfloat[]{
		\includegraphics[scale=0.54,page=12]{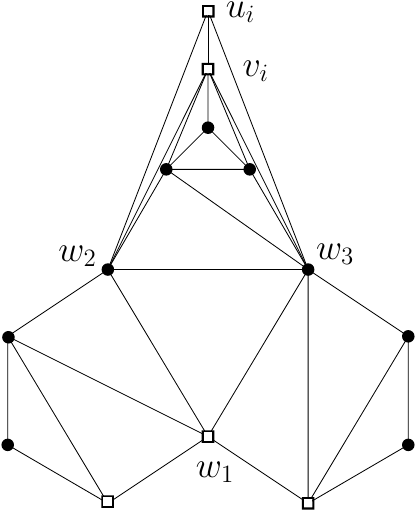}
		\label{fig:lemma7c}
	}\\
%	}
	\subfloat[]{
		\includegraphics[scale=0.54,page=13]{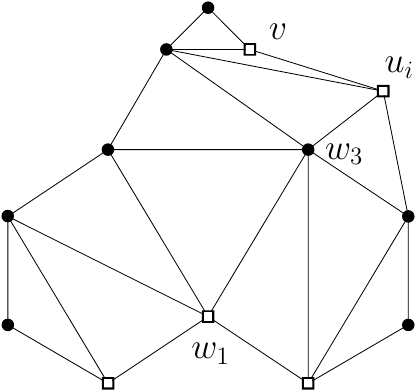}
		\label{fig:lemma7d}
	}~~
	\subfloat[]{
		\includegraphics[scale=0.54,page=14]{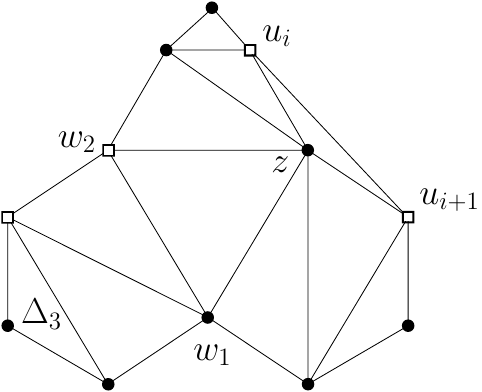}
		\label{fig:lemma7e}
	}~~
	\subfloat[]{
		\includegraphics[scale=0.54,page=15]{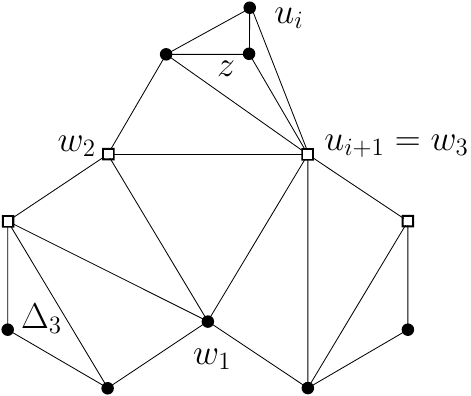}
		\label{fig:lemma7f}
	}
	\caption{
		Illustrating Lemma~\ref{lem:ExceptionCases}. In each case, the squared vertices form a TDS of $T$. (a) The graph $H_1$. (b) Removing the vertex $u_i$. (c) Removing the vertices $u_i$ and $v_i$. (d) The vertex $w_3$ is the vertex obtained by contracting the edge $(u_i,v_i)$. (e) and (f) Removing the edge $(u_i, u_{i+1})$.
	}
	\label{fig:lemma7}
\end{figure}

i) Suppose that $T=H_1$. If $u_i$ belongs to $\Delta$, say $u_i=w_1$, then $w_1$, its two neighbors in $C$ and two arbitrary vertices in $\Delta _1$ form a TDS $D$ (see Figure~\ref{fig:lemma7a}). If $u_i$ belongs to one of $\Delta _1, \Delta _2$ and $\Delta _3$, say $\Delta _1$, then $D$ is also a TDS by choosing $u_i$ as one of the vertices of $\Delta_1$ in $D$.

ii) Suppose that $T-u_i$ is $H_1$.
In this case, $u_i$  has at least two neighbors in $T$ that necessarily are consecutive vertices on the boundary of $H_1$ (see Figure~\ref{fig:lemma7b}). Hence, $u_i$ has a neighbor $v$ in one of the triangles $\Delta _1, \Delta _2$ and $\Delta _3$, say triangle $\Delta_1$. Then, $u_i$, $v$ and the three vertices of a TDS of the MOP $H_1-\Delta_1$ of order 9 define a TDS of $T$.

%iii) Suppose that $T-\{u_i,v_i\}$ is $H_1$ for some interior vertex $v_i$ adjacent to $u_i$. Hence, $v_i$ is adjacent to two consecutive vertices of the outerface of $H_1$. Therefore, $v_i$ is adjacent to some vertex $z$ in  $\Delta_i$, $i\in \{1,2,3\}$, say $\Delta_1$. Then, $v_i$, $z$ and the three vertices of a TDS of the MOP $H_1-\Delta_1$ of order 9 define a TDS of $T$.

iii) Suppose that $T-\{u_i,v_i\}$ is $H_1$ for some interior vertex $v_i$ adjacent to $u_i$. Assume first that $u_i$ has a neighbor $v$ in one of the triangles $\Delta _1, \Delta _2$ and $\Delta _3$, say triangle $\Delta_1$. As in the previous case, $u_i$, $v$ and the three vertices of a TDS of the MOP $H_1-\Delta_1$ of order 9 define a TDS of $T$.

Assume now that none of the vertices in $\Delta _1, \Delta _2$ and $\Delta _3$ is adjacent to $u_i$. In this case, since $T$ is a near-triangulation and $u_i$ a boundary vertex, then $v_i$ must be adjacent to all the vertices of at least one of the triangles $\Delta _1, \Delta _2$ and $\Delta _3$, say $\Delta_1$ (see Figure~\ref{fig:lemma7c} for an example). Therefore, $u_i$, $v_i$, and the three vertices of a TDS of the MOP $H_1-\Delta_1$ define a TDS of $T$.

iv) Suppose that $H_1$ is obtained from $T$ by contracting an edge $e=(u_i,v_i)$ incident with $u_i$, and let $w$ be the new vertex obtained after contracting this edge. If $w$ is one of the vertices of $\Delta_1, \Delta_2$ or $\Delta_3$, say $\Delta_1$, then the set formed by $u_i, v_i$ and the three vertices of a TDS of $H_1-\Delta_1$ is a TDS of $T$.

On the contrary, suppose that $w$ is one of the vertices of $\Delta$, say $w_3$ (see Figure~\ref{fig:lemma7d} for an example). In this case, $u_i$ has a neighbor $v$ in $T$ belonging to either $\Delta_1$ or $\Delta_2$. Assume that $v$ belongs to $\Delta_1$. The set formed by $u_i, v$ and the three vertices of a TDS of $H_1-\Delta_1$ is a TDS of $T$.

II) Suppose that $T-e$ is $H_1$ for some edge $(u_i,u_{i+1})$.
Let $z$ be the third vertex of the triangle in $T$ containing $e$.
Then $z$ belongs to one of the triangles $\Delta$, $\Delta_1$, $\Delta_2$ or $\Delta_3$.
Suppose first that $z$ belongs to $\Delta$. We may assume that $z=w_3$ (see Figure~\ref{fig:lemma7e}). Then, $u_i$ belongs to $\Delta_1$ and $u_{i+1}$ to $\Delta_2$ or viceversa. According to Lemma~\ref{lem:pentagon}, there is a TDS $D$ of size 2 containing $w_2$ in the triangulated pentagon defined by $w_1$, $w_2$ and $\Delta_3$. Therefore, $u_i, u_{i+1}$ and $D$ define a TDS of size 4 in $T$.

Now suppose that $z$ belongs to one of the triangles  $\Delta_1$,  $\Delta_2$ or  $\Delta_3$, say $\Delta_1$ (see Figure~\ref{fig:lemma7f}). Thus,
one of the vertices of $\{u_i,u_{i+1}\}$ is the vertex of degree 2 of $\Delta_1$ and the other one is $w_2$ or $w_3$, say $w_3$. If $D$ is a TDS of size 2 containing $w_2$ in the triangulated pentagon defined by $w_1$, $w_2$ and $\Delta_3$, then $D$ together with $w_3$ and a vertex in $\Delta_2$ adjacent to $w_3$ form a TDS of size 4 in $T$.
\end{proof}

To finish this section, we give some bounds on the size of a (total) dominating set of a near-triangulation under the contraction operation. Given a simple graph $G=(V,E)$, we say that a vertex $v\in V$ dominates a vertex $u\in V$ if $v$ and $u$ are adjacent in $G$. Thus, a vertex $v\in V$ dominates all its neighbors in $G$ but not itself.

\begin{lemma}\label{StandardLemma}
Let $T$ be a near-triangulation of order $n\ge 5$ with boundary cycle $C=(u_1, \ldots ,u_h,u_1)$. Suppose that for some vertex $u_i$ there is a contractible edge $e=(u_i,v_i)$ of $T$ such that $T/e$ has a TDS of size $s$. Then:
\begin{itemize}
	\item[I)] $T$ has a set of vertices $D$ satisfying one of the following conditions:
	\begin{itemize}
     \item[i)] $D$ is a TDS of size $s+1$ in $T$ such that $u_i$ and $v_i$ belong to $D$,
     \item[ii)] $D$ is a set of vertices of size $s$ such that neither $u_i$ nor $v_i$ belong to $D$ and $D$ dominates all vertices of $T$ except possibly one of $u_i$ or $v_i$.
    \end{itemize}
	\item[II)] There is a dominating set $D$ of size $s+1$ in $T$ such that $D$ contains $u_i$ and either $D$ is a TDS of $T$ or $D$ dominates all vertices of $T$ except possibly $u_i$.
\end{itemize}
\end{lemma}
\begin{proof}

I) The result follows from the same well-known result for abstract graphs: If $G/e$ is the graph obtained by contracting an edge $e=(u_i,v_i)$ of $G$ to a new vertex $w$, according to whether $w$ belongs to a TDS $D'$ of size $s$ in $G/e$ or not, either i) the set $D=\{ D'-w\} \cup \{u_i,v_i\} $ is a TDS of $G$ or ii) $D=D'$ dominates all vertices of $G$ except possibly $u_i$ or $v_i$.

II) As before, if the new vertex $w$ belongs to a TDS $D'$ of size $s$ in $G/e$, then $D=\{ D'-w\} \cup \{u_i,v_i\} $ is a TDS of $G$. Otherwise, the set $D = D' \cup \{u_i\}$ dominates all vertices of $T$ except possibly $u_i$.
\end{proof}

\section{Upper bound for near-triangulations}\label{sec:bound}

In this section we prove the main result of this paper: the upper bound $\lfloor \frac {2n}{5} \rfloor$ on the total domination number in near-triangulations of order $n$. Before proving it, we define the two main concepts required in its proof: reducible near-triangulations and terminal polygons.

Let $T$ be a near-triangulation with some interior vertices and boundary cycle $C=(u_1,u_2,\ldots,u_h,u_1)$. We say that $T$ is {\em reducible} if it contains a triangle $(u_i,u_{i+1},v)$ with $v$ a vertex not in $C$. In this case, by removing the boundary edge $u_iu_{i+1}$, we obtain a new near-triangulation $T'$ with boundary cycle  $C'=(u_1,\ldots,u_i,v,u_{i+1},\ldots ,u_h,u_1)$. Obviously, $\gamma _t(T)\le \gamma _t(T')$, and $T'$ contains fewer interior vertices than $T$. If $T'$ is also reducible, then we can obtain a new near-triangulation $T''$ with fewer interior points than $T'$. Iterating this process, we reach either a near-triangulation without interior vertices (a MOP), or a near-triangulation with interior vertices that is {\em irreducible}, that is, a near-triangulation with interior vertices such that for every boundary edge $(u_i, u_{i+1})$, the vertex $v$ in the triangular face $(u_i,u_{i+1},v)$ adjacent to $(u_i,u_{i+1})$ is also in $C$.  The simplest irreducible near-triangulation $H$ has order 7 and is shown in Figure~\ref{fig:NonreducibleBasic}. With these definitions, note that if $T$ is a near-triangulation, then $T$ is either reducible, or irreducible or a MOP.

\begin{figure}[ht]
	\centering
	\subfloat[]{
		\includegraphics[scale=0.56,page=16]{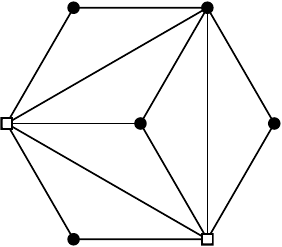}
		\label{fig:NonreducibleBasic}
	}~~~~~~~~~
	\subfloat[]{
		\includegraphics[scale=0.56,page=17]{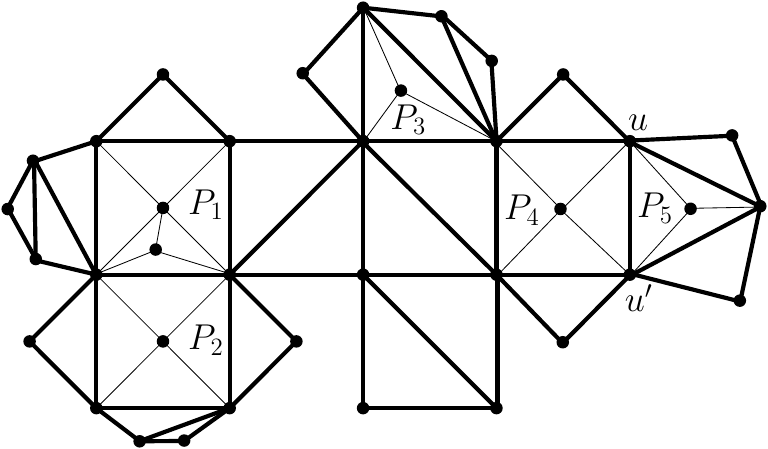}
		\label{fig:Nonreducible}
	}
	\caption{
		(a) The simplest irreducible near-triangulation $H$. The squared vertices form a total dominating set. (b) A irreducible near-triangulation $T$. Thick lines correspond to the subgraph $T[C]$. The diagonals of $T[C]$ define a set of adjacent polygons, five of which are non-empty and three are terminal, $P_2,P_3$ and $P_5$.
	}
	\label{fig:polygons}
\end{figure}

Let $T$ be a irreducible near-triangulation with boundary cycle  $C=(u_1,u_2,$ $\ldots,u_h,$ $u_1)$. The diagonals of the subgraph $T[C]$ divide the interior of $C$ into several regions whose interiors are disjoint. These regions are simple polygons that can be non-empty or empty, depending on whether they contain interior vertices of $T$ or not (see Figure~\ref{fig:Nonreducible}). Let $P_1,\ldots P_k$ denote the  polygons obtained in this way such that they contain some interior vertex of $T$.
The irreducible near-triangulation shown in Figure~\ref{fig:Nonreducible} contains five non-empty polygons $P_1, P_2, P_3, P_4$ and $P_5$. Observe that, by definition, every side $d$ of a polygon $P_i$ has to be a diagonal of $T[C]$, and that $P_i$ has no diagonals.
Therefore, a side $d$ of a polygon $P_i$ divides $T$ into two non-empty near-triangulations $T_{in}(P_i,d)$ and $T_{out}(P_i,d)$ sharing $d$, where $T_{in}(P_i,d)$ denotes the near-triangulation containing the polygon $P_i$. In Figure~\ref{fig:Nonreducible}, $T_{in}(P_5,(u,u'))$ is the near-triangulation of order 6 containing $P_5$.

We say that a non-empty polygon $P_i$ is {\em terminal} if at most one of the near-triangulations $T_{out}(P_i,d)$ corresponding to the sides $d$ of $P_i$ (diagonals in $T[C]$) contains interior vertices. Hence, if $P_i$ is a terminal polygon with $k$ sides, then at least $k-1$ of the near-triangulations $T_{out}(P_i,d)$ are MOPs with at least three vertices. The irreducible near-triangulation shown in Figure~\ref{fig:Nonreducible} contains three terminal polygons  $P_2,P_3,P_5$. The following lemma shows that a irreducible near-triangulation has at least one terminal polygon.

\begin{lemma}\label{lem:terminal}
Let $T$ be a irreducible near-triangulation of order $n\ge 7$ with boundary cycle $C$. Then, $T$ contains at least one terminal polygon.
\end{lemma}
\begin{proof}
Since $T$ is irreducible, it must contain non-empty polygons. Consider the dual graph $G=(V,E)$ associated with $T[C]$, where the vertices of $G$ are the faces defined by $T[C]$ and two vertices are adjacent in $G$ if their corresponding faces are adjacent. Since $T[C]$ is a Hamiltonian outerplane graph, $G$ must be a tree. Note that each non-empty polygon of $T[C]$ is a vertex of $G$.

If there is only one non-empty polygon, then it is terminal. Otherwise, observe that terminal polygons correspond to the leaves of the minimal subtree of $G$ containing all the vertices corresponding to non-empty polygons. Since every non-trivial tree has at least two leaves, then the lemma follows.
\end{proof}

We are now ready to prove the main result of the paper, Theorem~\ref{the:bound}. To this end, we also need the following two lemmas. The first one was proved in~\cite{ORourke83,Shermer91} and the proof of the second one is straightforward.

\begin{lemma}[\cite{ORourke83, Shermer91}]\label{lem:diagonal}
Given a MOP $G$ of order $n\ge 10$ and a boundary edge $(u_i,u_{i+1})$ of $G$, there exists a diagonal $d$ of $G$ that partitions $G$ into two MOPs, one of which contains exactly $6,7,8$ or $9$ vertices of $G$ and does not contain $(u_i,u_{i+1})$.
\end{lemma}

\begin{lemma}\label{lem:fn}
Let $n$, $k$, $d$ be positive integers. If $n-k\ge 5$ and $d/k\le 2/5$, then $\lfloor \frac{2(n-k)}{5}\rfloor  + d\le \lfloor \frac{2n}{5}\rfloor$.
\end{lemma}

\begin{theorem}\label{the:bound}
If $T=(V,E)$ is a near-triangulation of order $n\ge 5$, with boundary cycle $C=(u_1, \ldots , u_h,u_1)$, then $\gamma _t(G)\le \lfloor \frac{2n}{5}\rfloor $
except if $T$ is $H_1$ or $H_2$.
\end{theorem}

\begin{proof}
By convenience, we define $f(n)$ as $\lfloor \frac{2n}{5}\rfloor $. Thus, $f(n-k)+d \le f(n)$ if $n-k\ge 5$ and $d/k\le 2/5$.

We proceed by induction on the number $m$ of interior vertices of $T$ and the number $n$ of vertices of $T$.
For $m=0$, the base of the induction, $T$ is a MOP and the result is true by Theorem~\ref{the:MOPs}.

Let $T$ be a near-triangulation of order $n$, with $m>0$ interior vertices and boundary cycle $C=(u_1, \ldots , u_h,u_1)$. Suppose that $\gamma _t(T')\le f(n')$ for any near-triangulation $T'$ of order $n'\ge 5$ such that either $T'$ is different from $H_1,H_2$ and contains $m'<m$ interior vertices, or $T'$ contains $m'=m$ interior vertices and $n'<n$. We need to prove that $\gamma _t(T)\le f(n)$.

To make further reasoning easier, we prove the following claim.

\begin{claim}\label{Claim}
Let $T$ be a near-triangulation of order $n\ge 6$, with $m$ interior vertices and with boundary cycle $C=(u_1, \ldots , u_h,u_1)$. Assume that the previous induction hypotheses hold, that is, $\gamma _t(T')\le f(n')$ for any near-triangulation $T'$ of order $n'\ge 5$ such that either $T'$ is different from $H_1,H_2$ and contains $m'<m$ interior vertices, or $T'$ contains $m'=m$ interior vertices and $n'<n$. For any vertex $u_i\in C$, there exists a dominating set $D$ of size at most $f(n-1)+1$ such that $D$ contains $u_i$ and all the vertices of $T$ are dominated except possibly $u_i$.
\end{claim}
\noindent{\em Proof of the claim}.
Assume that $T$ is neither $H_1$ nor $H_2$. By Lemma~\ref{lem:ContractibleEdges}, there is always a contractible edge $e=(u_i,v_i)$ with $u_i$ as one of its endpoints. Note that $T/e$ has either fewer interior vertices than $T$ or the same number of interior vertices but $n-1$ vertices. Thus, if $T/e$ is not $H_1$ or $H_2$, then $\gamma_t(T/e) \le f(n-1)$ by the induction hypotheses. In this case, the result follows from Lemma~\ref{StandardLemma}(II). If $T/e$ is $H_1$ or $H_2$, then the order of $T$ is $13$ and the result follows from  Lemma~\ref{lem:ExceptionCases}(iv), since $f(12)+1=5$. Finally, if $T$ is $H_1$ or $H_2$, then  Lemma~\ref{lem:ExceptionCases}(i) ensures the result because $f(11)+1=5$. \hfill $\square$

\vskip 0.25 cm
Let us go into the details of the proof of the theorem. Assume first that $T$ is reducible. Hence, by removing a suitable boundary edge $(u_i,u_{i+1})$ we obtain a near-triangulation $T'$ of order $n$ with $m-1$ interior vertices. If $T'$ is $H_1$ or $H_2$, then $T$ has 12 vertices and Lemma~\ref{lem:ExceptionCases}(II) guarantees that $\gamma _t(T) = 4 = f(12)$. Otherwise, the induction hypothesis can be applied to $T'$, so $\gamma _t(T) \le \gamma _t(T') \le f(n)$.

Assume then that $T$ is irreducible, hence $n\ge 7$ and $T$ contains at least one terminal polygon $P$ by Lemma~\ref{lem:terminal}, with $k\ge 3$ sides $d_1=(u'_1,u'_2),d_2=(u'_2,u'_3),\ldots ,d_k=(u'_k,u'_1)$. Note that the vertices $u'_1, \ldots , u'_k$ of $P$ correspond to vertices in $C$ and that are in clockwise order. Without loss of generality, we may assume that $u'_1=u_1$. For $j=1,\ldots ,k$, every near-triangulation $T_{out}(P,d_j)=M_j$ is a MOP, except possibly one of them, say $T_{out}(P,d_k)=M_k$. Let $\overline{M_j}$ denote the near-triangulation $T_{in}(P,d_j)$, so $|M_j|+|\overline{M_j}|=n+2$, where $| \cdot |$ is the number of vertices of a graph.
%Observe that,  for $j=1,\ldots ,k-1$,  $\overline{M_j}$ is a reducible near-triangulation because $d_j$ can be removed from $\overline{M_j}$ (see Figure~\ref{fig:TerminalPolygon}). If $M_k$ is a MOP, then $\overline{M_k}$ is also a reducible near-triangulation.
Observe that, since $P$ is non-empty and has no diagonals,   $\overline{M_j}$ is a reducible near-triangulation for $j=1,\ldots ,k$, because $d_j$ can be removed from $\overline{M_j}$ (see Figure~\ref{fig:TerminalPolygon}).

We prove that $\gamma_t(T)\le f(n)$ by applying induction to a suitable near-triangulation obtained after some graph operations.
%(removing vertices or edges or contracting edges).
We distinguish cases according to the sizes of the MOPs $M_j$.

\vskip 0.5 cm

\noindent
{\bf Removing vertices from one MOP}

\vskip 0.5 cm

We begin analyzing the cases when there is a MOP $M_j$ such that either $|M_j|\in \{ 4,6,7,8\}$, or $|M_j|=9$ and $d_j$ is contractible in $\overline{M_j}$, or $|M_j|>9$.
These cases are the same as those described in~\cite{Lemanska17}, except for the case $|M_j|=4$, and the analysis is totally analogous. For the sake of completeness, we include them. Note that $\overline{M_j}$ contains interior vertices, so it is neither $H_1$ nor $H_2$, and has at least 6 vertices (because $T$ is irreducible). Therefore, the induction hypothesis can be applied on $\overline{M_j}$ if necessary.

\vskip 0.2 cm
\noindent
{\bf Case 1:} $|M_j|=4$.
\vskip 0.2 cm

Suppose that there is a MOP $M_j$ of order 4 ($M_5$ in Figure~\ref{fig:TerminalPolygon}). One of $u'_j$ or $u'_{j+1}$ is a dominating set of $M_j$ (the vertex $u'_6$ in Figure~\ref{fig:TerminalPolygon}). Suppose that $u'_{j+1}$ is such a vertex (the same reasoning can be applied in the other case). Note that $\overline{M_j}$ has $n-2$ vertices and is reducible because the edge $(u'_j,u'_{j+1})$ can be removed from $\overline{M_j}$. Let $(u'_{j+1},u_i)$ be the other boundary edge of $\overline{M_j}$ incident with $u'_{j+1}$.

From $\overline{M_j}$, we build another reducible near-triangulation $\overline{M_j}'$ of order $n$, by adding two vertices $w_1$ and $w_2$ and the edges $(u'_{j+1},w_1), (u'_{j+1},w_2), (w_2,w_1)$  and $(u_i,w_2)$ in the outer face, that is, a MOP of order 4 is joined to the edge $(u'_{j+1},u_i)$. Since $\overline{M_j}'$ is reducible, the induction hypothesis can be applied to $\overline{M_j}'$, so it has a TDS $D$ of size at most $f(n)$. Recall that Lemma~\ref{lem:ExceptionCases}(II) guarantees the same bound for $D$, even in the case that either $H_1$ or $H_2$ is obtained after the reduction.

From $D$, we build as follows another TDS $D'$ of $\overline{M_j}'$ such that $|D'| \le f(n)$, $D'$ contains $u'_{j+1}$ and does not contain either $w_1$ or $w_2$. The degree of $w_1$ in $\overline{M_j}'$ is 2, hence at least one of $u'_{j+1}$ and $w_2$ must belong to $D$ so that $w_1$ is dominated. Suppose that $u'_{j+1}$ belongs to $D$. If neither $w_1$ nor $w_2$ belongs to $D$, we are done. Otherwise, since the neighbors of $w_1$ and $w_2$ are also neighbors of $u'_{j+1}$,  by removing $w_1$ and $w_2$ from $D$ (at least one belongs to $D$) and by adding a neighbor of $u'_{j+1}$ to $D$ (if no neighbor of $u'_{j+1}$ different from $w_1$ and $w_2$ belongs to $D$), we obtain such a set $D'$. On the contrary, suppose that $u'_{j+1}$ does not belong to $D$ but $w_2$ does. Thus, by removing $w_2$ from $D$ and by adding $u'_{j+1}$ to $D$ (and removing $w_1$ and adding a neighbor of $u'_{j+1}$ different from $w_1$ and $w_2$ if $w_1$ belongs to $D$), such a set $D'$ is obtained. Since $u'_{j+1}$ dominates the vertices of $M_j$, then $D'$ is a TDS of $T$ and $\gamma _t(T)\le f(n)$.

\vskip 0.2 cm
\noindent
{\bf Case 2:} $|M_j|=6$.
\vskip 0.2 cm

Suppose that there is a MOP $M_j$ of order 6 ($M_3$ in Figure~\ref{fig:TerminalPolygon}). Since $M_j$ is a triangulated hexagon, by Lemma~\ref{lem:hexagon}, either $u'_j$ and one of its neighbors, or $u'_{j+1}$ and one of its neighbors form a TDS of the triangulated hexagon $M_j$. Assume that $\{u'_j,u\}$ is such a set (the other case is analyzed in the same way). By Claim~\ref{Claim}, $\overline{M_j}$ has a set $D$ of size at most $f(n-5)+1$ containing the vertex $u'_j$ and dominating all the vertices of $\overline{M_j}$ except possibly  $u'_j$. But then, the set $D\cup \{ u\} $ is a TDS of $T$ with size at most $f(n-5)+2=f(n)$.

\begin{figure}[tb]
\centering
\includegraphics[scale=0.50,page=18]{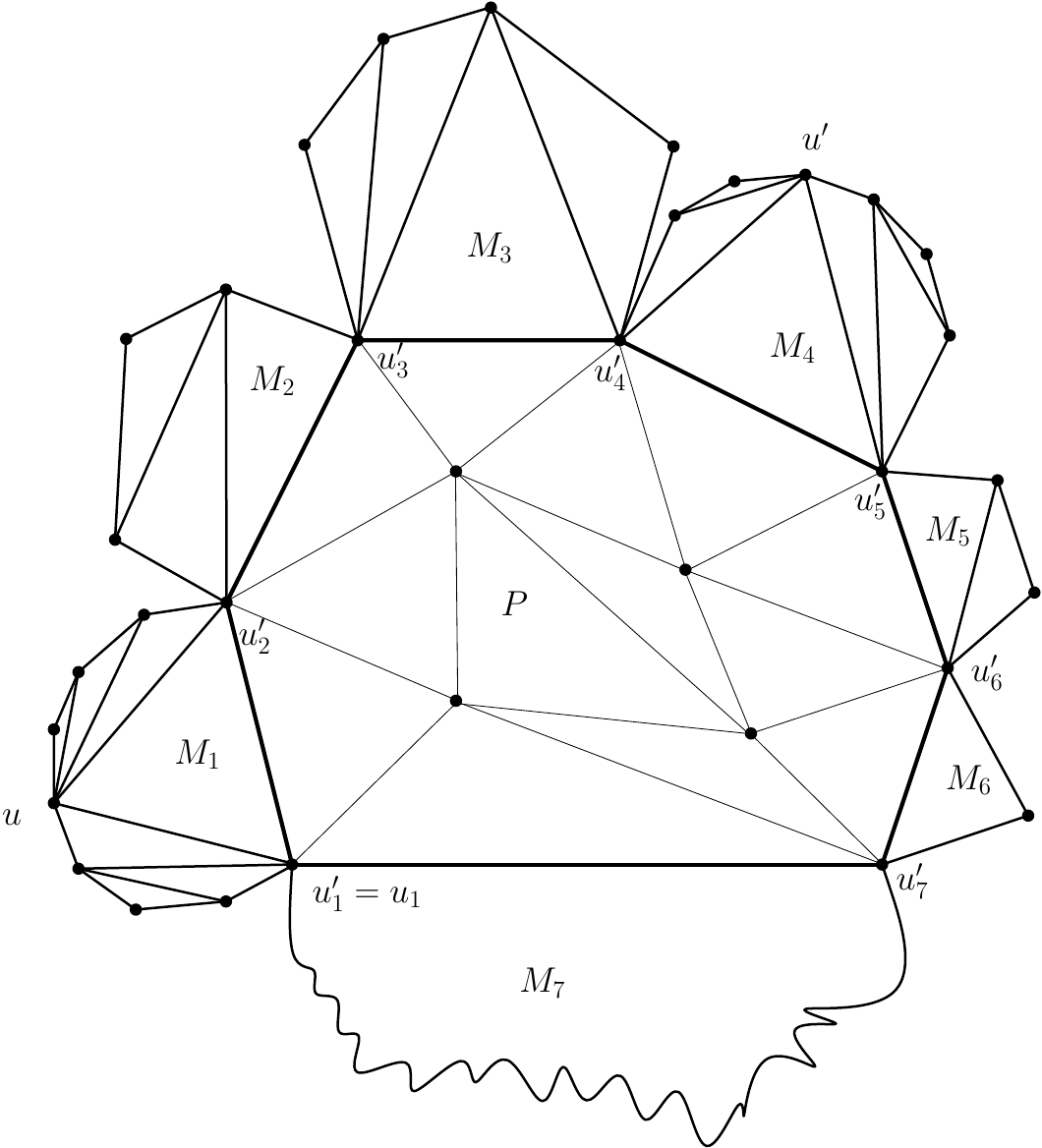}
\caption{A terminal $7$-gon $P$ with 6 MOPs $M_1, M_2, M_3, M_4, M_5$ and $M_6$ of orders $9,5,6,8,4$ and $3$, respectively, around it.
}\label{fig:TerminalPolygon}
\end{figure}

\vskip 0.2 cm
\noindent
{\bf Case 3:} $|M_j|=7$.
\vskip 0.2 cm

Suppose that there is a MOP $M_j$ of order 7. In this case, a TDS of $\overline{M_j}$ has size at most $f(n-5)$ by the induction hypothesis. This set can be transformed into a TDS of $T$ by adding a TDS of $M_j$ that consists of two vertices by Theorem~\ref{the:MOPs}. Therefore, $f(n-5)+2 = f(n)$, so $\gamma_t(T) \le f(n)$.

\vskip 0.2 cm
\noindent
{\bf Case 4:} $|M_j|=8$.
\vskip 0.2 cm

Suppose that there is a MOP $M_j$ of order 8 ($M_4$ in Figure~\ref{fig:TerminalPolygon}). Let $\{u'_j=u_k, \ldots ,$ $u_{k+7}=u'_{j+1}\}$ denote the vertices of $M_j$. Let $\Delta=(u'_j,u'_{j+1},u')$ be the triangle adjacent to the edge $(u'_j,u'_{j+1})$ in $M_j$. If $u'$ is $u_{k+1}$, $u_{k+2}$, $u_{k+5}$ or $u_{k+6}$, then either $(u'_j,u')$ or $(u'_{j+1},u')$ defines a MOP of order 6 or 7, and we can argue as in Cases 2 or 3, respectively.

Assume that $u'=u_{k+3}$ (the case $u'=u_{k+4}$ is symmetric). By removing the vertices $u_{k+1},u_{k+2},u_{k+4},u_{k+5},u_{k+6}$ from $T$, we obtain a new near-triangulation $T'$ of order $n-5\ge 7$ and $m$ interior vertices. By the induction hypothesis, $T'$ has a TDS $D'$ of size at most $f(n-5)$ that necessarily contains either $u'_j$ or $u'_{j+1}$ since the degree of $u'$ in $T'$ is 2.

If $D'$ contains $u'_j$, then by adding $u'$ and a suitable vertex $v$ adjacent to $u'$ in the triangulated pentagon $\{u',u_{k+4}, u_{k+5}, u_{k+6}, u'_{j+1}\}$,
we obtain a TDS of $T$ with size at most $f(n-5)+2 =f(n)$. If $D'$ contains $u'_{j+1}$, applying Lemma~\ref{lem:pentagon} to the triangulated pentagons $\{u'_{j+1},u', u_{k+4}, u_{k+5}, u_{k+6}\}$ and $\{u'_{j+1},u'_j, u_{k+1}, u_{k+2}, u'\}$, we can then obtain a TDS in $T$ of size at most $f(n-5)+2$, by adding one additional vertex in each one of these two triangulated pentagons.

\vskip 0.2 cm
\noindent
{\bf Case 5:} $|M_j|=9$ and $d_j$ is contractible in $\overline{M_j}$.
\vskip 0.2 cm

Suppose that there is a MOP $M_j$ of order 9 ($M_1$ in Figure~\ref{fig:TerminalPolygon}).
Let $\Delta=(u'_j,u'_{j+1},u)$ be the triangle adjacent to the edge $(u'_j,u'_{j+1})$ in $M_j$ and let $\{u'_j=u_k, \ldots ,$ $u_{k+8}=u'_{j+1}\}$ denote the vertices of $M_j$. If $u$ is $u_{k+1}$, $u_{k+2}$, $u_{k+3}$, $u_{k+5}$, $u_{k+6}$ or $u_{k+7}$, then either $(u'_j,u)$ or $(u'_{j+1},u)$ defines a MOP of order 6, 7 or 8, and we can argue as in Cases 2, 3 or 4, respectively.

Assume that $u=u_{k+4}$. In this case, the sets of vertices $\{u'_j, u_{k+1}, u_{k+2}, u_{k+3}, u\}$ and $\{u, u_{k+5}, u_{k+6}, u_{k+7}, u'_{j+1}\}$ induce two triangulated pentagons. Since $d_j$ is contractible in $\overline{M_j}$, then $\overline{M_j}/ d_j$ is a near-triangulation of order $n-8\ge 5$ with $m$ interior vertices. Thus, $\overline{M_j}/ d_j$ is different from $H_1,H_2$ and has a TDS of size at most $f(n-8)$ by the induction hypothesis.

As a consequence, by Lemma~\ref{StandardLemma}(I), $\overline{M_j}$ has either a TDS $D$ of size at most $f(n-8)+1$ containing $u'_j$ and $u'_{j+1}$, or a set $D$ of size at most $f(n-8)$, not containing either $u'_j$ or $u'_{j+1}$, and dominating every vertex of $\overline{M_j}$ except possibly $u'_j$ or $u'_{j+1}$. In the first case, by  Lemma~\ref{lem:pentagon} we can add to $D$ a suitable vertex in each one of the two previous triangulated pentagons, so that the resulting set is a TDS of $T$ of size at most $f(n-8)+3\le f(n)$. In the second case, by Theorem~\ref{the:MOPs}, there is a TDS $D'$ of size 3 in $M_j$. Therefore, $D\cup D'$ is a TDS in $T$ of size $f(n-8)+3 \le f(n)$.

\vskip 0.2 cm
\noindent
{\bf Case 6:} $|M_j|>9$.
\vskip 0.2 cm

Suppose that there is a MOP $M_j$ of order greater than 9. By Lemma~\ref{lem:diagonal}, there is a diagonal $d$ in $M_j$ such that it partitions $M_j$ into two MOPs, one of which, $M'$, has $6,7,8$ or $9$ vertices and does not contain the edge $(u'_j,u'_{j+1})$. Therefore, we can also argue as in Cases 2, 3, 4 and 5 by removing $M'$ from $T$, since $d$ is contractible in the near-triangulation obtained after removing $M'$.

\vskip 0.5 cm

\noindent
{\bf Removing vertices from two or more MOPs}

\vskip 0.5 cm

We now study irreducible near-triangulations where all MOPs $M_j$ are of order $3,5$ or $9$. Besides, the case of a MOP $M_j$ of order 9 must be analyzed only when $d_j$ is not contractible in $\overline{M_j}$. In this situation,
we have to remove vertices from more than one MOP. Most of the cases can be solved by removing vertices from two consecutive MOPs $M_j$ and $M_{j+1}$ around
the terminal polygon $P$. We recall that $M_k$ can be a MOP or not. If it is not a MOP, then $|M_k|\ge 6$  because $T$ is irreducible and the graph $H$ shown in Figure~\ref{fig:NonreducibleBasic} without a vertex of degree 2 is the simplest graph that can be adjacent to $d_k$.
%(because it must contain interior vertices) and $T$ is irreducible.
If it is a MOP, we can assume that it is the largest one, among all MOPs $M_j$ adjacent to $P$ (by renumbering them if necessary).

If there exist at least two MOPs of different sizes, then we can assume that there are two consecutive MOPs $M_j$ and $M_{j+1}$ such that $\{|M_j|,|M_{j+1}|\}$ are either $\{5,3\}$, or $\{9,3\}$, or $\{9,5\}$
%(the cases $\{3,5\}$, $\{3,9\}$ and $\{5,9\}$ can be studied in a similar manner).
Otherwise, all the MOPs are of order either $3$ or $5$ or $9$.
%For the sake of clarity in further reasoning, we assume that the two consecutive MOPs $M_j$ and $M_{j+1}$ of different size, when they exist, are $M_1$ and $M_2$, respectively.
For the sake of clarity and since, as we will see, the reasoning used in the proof holds for every pair of consecutive MOPs of different order, we assume that these MOPs of different order, whenever they exist, are $M_1$ and $M_2$ and that $|M_1| > |M_2|$.

Let $\overline{M}$ denote the near-triangulation obtained by removing from $T$ the vertices of $M_1$ and $M_2$ that are not in $P$. Hence, $|\overline{M}|=n-(|M_1|-2)-(|M_{2}|-2)$ and $(u'_1,u'_2)$ and $(u'_2,u'_3)$ are boundary edges of $\overline{M}$.
Observe that $|\overline{M}| \ge 2+|M_k|$, since $P$ is at least a triangle containing at least one interior vertex and $M_k$ is included in $\overline{M}$.
Therefore, the induction hypothesis can be applied to $\overline{M}$ when necessary, since $|\overline{M}| \ge 5$ and it is neither $H_1$ nor $H_2$ ($\overline{M}$ contains interior vertices).
Next, we analyze all possible combinations of the sizes of $M_j$'s.

\vskip 0.2 cm
\noindent
{\bf Case 7:} $|M_1|=5$ and $|M_2|=3$.
\vskip 0.2 cm

Since $M_1$ is a triangulated pentagon, $M_1$ has a TDS formed by the vertex $u'_2$ and one of its neighbors $u'$ by Lemma~\ref{lem:pentagon} (see Figure~\ref{fig:case8}). Besides, $P$ does not contain diagonals, so there is no diagonal incident to $u'_2$ in $\overline{M}$. By Lemma~\ref{lem:deletion}, $\overline{M}-\{ u'_2\} $ is a near-triangulation of order $n-5$. Recall that if $M_k$ is a MOP, then $|M_k|\ge 5$ and if it is not a MOP, then $|M_k|\ge 6$. %($T$ is irreducible so the graph $H$ shown in Figure~\ref{fig:NonreducibleBasic} without a vertex of degree 2 is the simplest graph that can be adjacent to $d_k$).
As a consequence,  $n-5 =|\overline{M}| -1 \ge |M_k| + 1 \ge 6$, because $P$ contains an interior vertex, and the induction hypothesis can be applied on $\overline{M}-\{ u'_2\} $.

Suppose that $\overline{M}-\{ u'_2\} $ is neither $H_1$ nor $H_2$, so it has a TDS $D$ of size at most $f(n-5)$ by the induction hypothesis. Thus, $D\cup \{ u'_2,u'\} $ is a TDS of $T$ of size at most $f(n-5)+2 = f(n)$. On the contrary, if $\overline{M}-\{ u'_2\} $ is either $H_1$ or $H_2$, then  Lemma~\ref{lem:ExceptionCases}(ii) guarantees that $\overline{M}$ has a TDS $D'$ of size 5 containing $u'_2$.  Therefore, $D'\cup \{ u'\} $ is a TDS in $T$ of size 6, so $\gamma_t(T) \le f(n)$ since the order of $T$ is $17$ and $f(17) = 6$.

\vskip 0.2 cm
\noindent
{\bf Case 8:} $|M_1|=9$, $|M_2|=3$ and $d_1=(u'_1,u'_2)=(u_1,u_9)$ is not contractible.
\vskip 0.2 cm

Arguing as in Case 5, we may assume that $\Delta=(u_1,u_9,u_5)$ is the triangle adjacent to the edge $(u_1,u_9)$ in $M_1$, because otherwise a MOP of order 6,7 or 8 could be removed. Thus, the vertices $\{u_1, u_2, u_3, u_4, u_5\}$ and $\{u_5, u_6, u_7, u_8, u_9\}$ induce two triangulated pentagons, $P'$ and $P''$, respectively (Figure~\ref{fig:case9}). Applying Lemma~\ref{lem:pentagon} to $P'$ and $P''$, there exist two vertices $u'\in P'$ and $u''\in P''$ such that $\{u_5,u',u''\}$ is a TDS of $M_1$.

Since $d_1=(u'_1,u'_2)$ is not contractible in $\overline{M_1}$, then it is also not contractible in the near-triangulation $T'$ induced by the vertices of the terminal polygon $P$ and the vertices inside $P$. $T'$ has no diagonals, hence there exists a vertex $v_2$ inside $P$ by Lemma~\ref{lem:RemovingPoints}(iii), such that $v_2$ is adjacent to $u'_2=u_9$ and $T'-\{ u'_2,v_2\} $ is a near-triangulation. As a consequence, $\overline{M}-\{ u'_2,v_2\} $ is a near-triangulation of order $n-10\ge 7$ (recall that $|M_k|\ge 6$). If $\overline{M}-\{ u'_2,v_2\} $ is neither $H_1$ nor $H_2$, then it has a TDS $D$ of size at most $f(n-10)$ by the induction hypothesis. Thus, $D\cup \{ u_5,u',u'',u'_2\} $ is clearly a TDS of $T$ with size at most $f(n-10)+4=f(n)$, so $\gamma_t(T)\le f(n)$. On the contrary, if $\overline{M}-\{ u'_2,v_2\} $ is either $H_1$ or $H_2$, then $n=22$ and Lemma~\ref{lem:ExceptionCases}(I) (iii) ensures that there exists a TDS  $D$ containing $u'_2$ of size $5$ in $\overline{M}$. The set $D\cup \{ u_5,u',u''\} $ is a TDS of $T$ with size $8=f(22)$.

\begin{figure}[ht]
	\centering
	\subfloat[]{
		\includegraphics[scale=0.56,page=19]{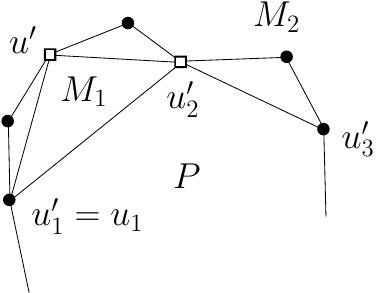}
		\label{fig:case8}
	}~
	\subfloat[]{
		\includegraphics[scale=0.56,page=20]{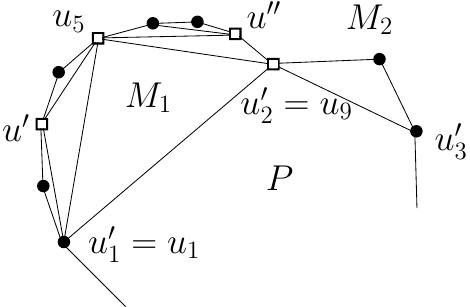}
		\label{fig:case9}
	}~
	\subfloat[]{
		\includegraphics[scale=0.56,page=21]{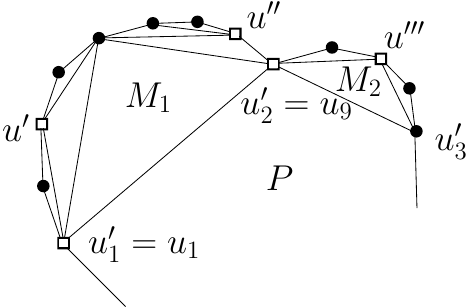}
		\label{fig:case10}
	}\\
	\subfloat[]{
		\includegraphics[scale=0.56,page=22]{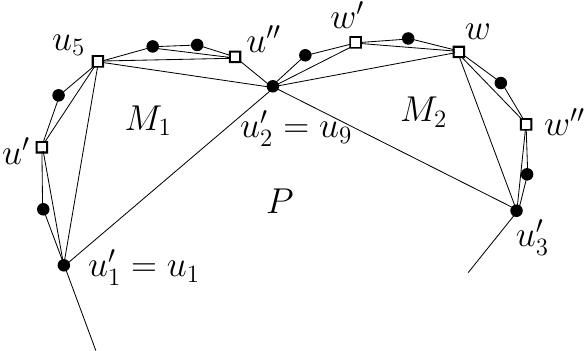}
		\label{fig:case11}
	}~
	\subfloat[]{
		\includegraphics[scale=0.56,page=23]{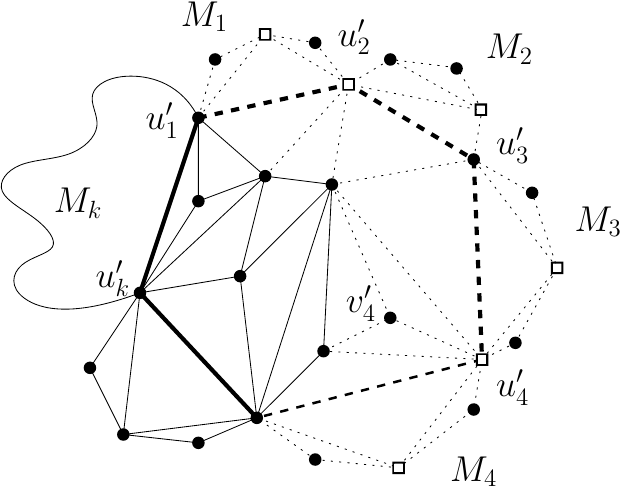}
		\label{fig:case12}
	}
	\caption{%Removing vertices of several MOPs. %The squared vertices form TDSs of the removes MOPs.
(a) Case 7: $u'_2$ and $u'$ define a TDS in $M_1$. (b) Case 8: $u_5, u'$ and $u''$ form a TDS in $M_1$. (c) Case 9: $u'_1, u'_2, u', u''$ and $u'''$ are a TDS in $M_1\cup M_2$. (d) Case 11: $u_5, u', u'', w, w'$ and $w''$ form a TDS in $M_1\cup M_2$. (e) Case 12: Removing the MOPs $M_1, M_2, M_3$ and $M_4$, and the vertices $u'_2, u'_3, u'_4$ and $v'_4$ to obtain the near-triangulation $\overline{M}$. The squared vertices form a TDS of $\{M_1\cup M_2\}\cup \{M_3\cup M_4\}$.
		}
	\label{fig:2MOPcases}
\end{figure}

\vskip 0.2 cm
\noindent
{\bf Case 9:} $|M_1|=9$, $|M_2|=5$ and $d_1=(u'_1,u'_2)=(u_1,u_9)$ is not contractible.
\vskip 0.2 cm

Arguing as in Case 8, we may assume that $\Delta=(u_1,u_9,u_5)$ is the triangle adjacent to the edge $(u_1,u_9)$ in $M_1$ (so  $\{u_1, u_2, u_3, u_4, u_5\}$ and $\{u_5, u_6, u_7, u_8, u_9\}$ induce two triangulated pentagons, $P'$ and $P''$), and that $\overline{M}-\{ u'_2,v_2\} $ is a near-triangulation of order $n-12\ge 7$.

By Claim~\ref{Claim}, $\overline{M}-\{ u'_2,v_2\} $ has a set $D$ of size $\le f(n-13)+1$ containing the vertex $u_1$ and dominating all the vertices of $\overline{M}-\{ u_9,v_2\} $ except possibly $u_1$. We add $u'_2$ to $D$ and, by Lemma~\ref{lem:pentagon}, we can also add to $D$ a vertex $u'$ to dominate $P'$, a vertex $u''$ to dominate $P''$ and a vertex $u'''$ to dominate $M_2$ (see Figure~\ref{fig:case10}). Therefore, $D\cup \{u'_2,u',u'',u'''\}$ is a TDS of $T$ with size at most $f(n-13)+5\le f(n)$.

\vskip 0.2 cm
\noindent
{\bf Case 10:} All MOPs $M_j$ are of order 3, so $|M_1|=|M_2|=3$.
\vskip 0.2 cm

This case is similar to Case 1. $\overline{M}$ is reducible (any of $(u'_1,u'_2)$ and $(u'_2,u'_3)$ can be removed), hence the graph $\overline{M}'$ of order $n$, obtained from $\overline{M}$ by adding two vertices $w_1,w_2$ in the outer face and the edges $(u'_2,w_1),(u'_2,w_2),(w_2,w_1),(u'_3,w_2)$, is also reducible by removing $(u'_1,u'_2)$. Arguing as in Case 1, $\overline{M}'$ has a TDS $D'$ of size at most $f(n)$ containing the vertex $u'_2$ and not containing either $w_1$ or $w_2$, even in the case that $\overline{M}'$ is reducible to either $H_1$ or $H_2$\footnote{In fact, a detailed analysis of cases shows that $\overline{M}'$ cannot be either $H_1$ or $H_2$.}. Therefore, $\gamma_t(T)\le f(n)$ since $D'$ is also a TDS of $T$.

\vskip 0.2 cm
\noindent
{\bf Case 11:} All MOPs $M_j$ are of order 9 and all $d_j$ are not contractible.
\vskip 0.2 cm

We have $|M_1| = |M_2| = 9$ and $d_1$ and $d_2$ are not contractible. As in case 8, we may assume that $\Delta=(u_1,u_9,u_5)$ is the triangle adjacent to the edge $(u_1,u_9)$ in $M_1$, so $\{u_1, u_2, u_3, u_4, u_5\}$ and $\{u_5, u_6, u_7, u_8, u_9\}$ induce two triangulated pentagons, $P'$ and $P''$. Therefore, there exist two vertices $u'\in P'$ and $u''\in P''$ such that $D_1=\{u_5,u',u''\}$ is a TDS of $M_1$. The same happens in $M_2$, so $M_2$ has a TDS $D_2=\{w,w',w''\}$ of size 3 (see Figure~\ref{fig:case11}).

Since $P$ contains no diagonals, $\overline{M}'=\overline{M}-\{ u_9\} $ is a near-triangulation of order $n-15\ge 7$ by Lemma~\ref{lem:deletion}. We claim that $\overline{M}'$ is  neither $H_1$ nor $H_2$. We recall that $\overline{M}'$ must contain $M_k$. If $M_k$ is not a MOP, then it contains interior vertices, so $\overline{M}'$ is  neither $H_1$ nor $H_2$. Assume to the contrary that $M_k$ is a MOP, so $|M_k|\ge 9$ by hypothesis, and that $\overline{M}'$ is $H_1$ (the same reasoning applies if $\overline{M}'$ is $H_2$). $P$ is terminal, hence some vertices of $H_1$ must be interior vertices in $\overline{M}$, implying that $d_k$ is a diagonal of $H_1$. Thus, by the symmetry of $H_1$ (see Figure~\ref{fig:ExceptionCases}), $d_k$ can only be one of the edges $(3,7),(3,6)$ and $(4,6)$. If $d_k$ is $(3,6)$ or $(4,6)$, then it defines a MOP of size at least 10 and we are in Case 6. If $d_k$ is $(3,7)$, then it defines a MOP of size 9, where $(3,7)$ would be  contractible in $\overline{M_k}$ and we would be in Case 5. Hence, $\overline{M}'$ is  neither $H_1$ nor $H_2$.

As a consequence, $\overline{M}'$ has a total dominating set $D$ of size at most $f(n-15)$ by the induction hypothesis. Therefore, $D\cup D_1 \cup D_2$ is a TDS in $T$ of size at most $f(n-15)+6= f(n)$.

\vskip 0.2 cm
\noindent
{\bf Case 12:} All MOPs $M_j$ are of order 5.
\vskip 0.2 cm

The case $|M_j|=5$ for every MOP $M_j$ is the only case left. We recall that $(u'_1,u'_2),$ $,\ldots , (u'_k,u'_1)$ denote the diagonals $d_1, \ldots , d_k$ of $T$ defining the terminal polygon $P$, and that $M_k$ can also be a MOP when $P$ is the only non-empty polygon of $T$. If it is the case, then $M_k$ must also have 5 vertices. Next, we explain how to obtain a TDS of size at most $f(n)$, by removing vertices from several consecutive MOPs.

Let $T'$ be the near-triangulation induced by $P$ and its interior vertices.
We distinguish whether $T'$ has one interior vertex or more than one.

Assume first that $T'$ has at least two interior vertices.  By  Lemma~\ref{lem:RemovingPoints}(ii), there is a vertex $u'_j$ in $T'$, $2\le j<k$, and an interior vertex $v'_j$ adjacent to $u'_j$ such that the graph $T'-\{ u'_2,\ldots ,u'_j,v'_j\} $ is a near-triangulation. As a consequence, by removing the vertices in the MOPs $M_1,M_2,\ldots ,M_{j}$ that do not belong to $P$, and the vertices $u'_2,u'_3,\ldots ,u'_j,v'_j$, we obtain a near-triangulation $\overline{M}'$ of size $|\overline{M}'|=n-3j-j=n-4j\ge 6$ (see Figure~\ref{fig:case12}). Since every MOP $M_i$ is a triangulated pentagon, observe that the vertex $u'_i$,  $2\le i\le j$, a neighbor $v_{i-1}$ of $u'_i$ in $M_{i-1}$ and another neighbor $v_i$ of $u'_i$ in $M_i$, form a TDS of size 3 of $M_{i-1}\cup M_i$.

\begin{figure}[ht]
	\centering
	\subfloat[]{
		\includegraphics[scale=0.80,page=24]{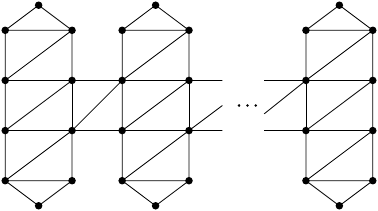}
		\label{fig:tightMOP}
	}~~~~~
	\subfloat[]{
		\includegraphics[scale=0.50,page=25]{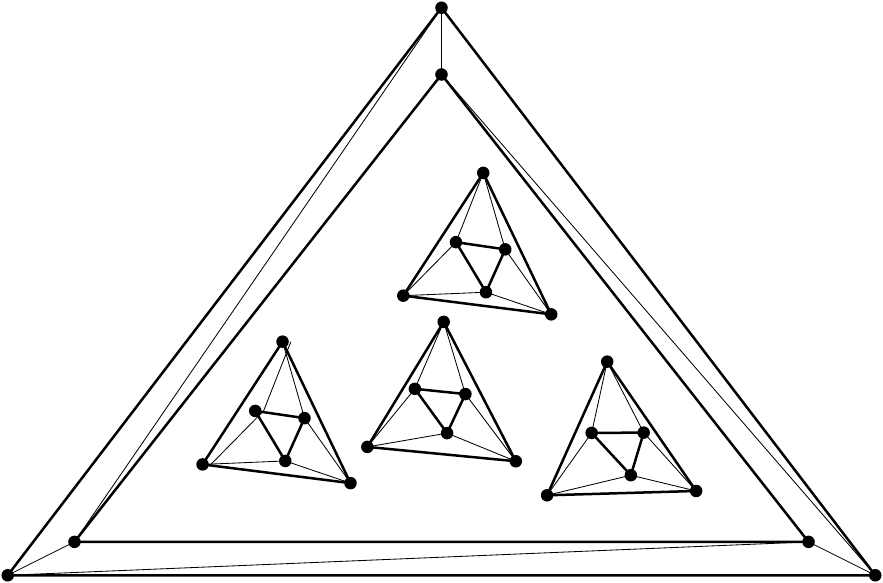}
		\label{fig:tight}
	}
	\caption{
		(a) A MOP $T$ of order $n$ such that $\gamma_t(T) = \lfloor \frac{2n}{5}\rfloor$. Any TDS must contain at least two vertices of each MOP of order 5. (b) Triangulating in any way the inter-octahedra region, a triangulation $T$ of order $n$ is obtained such that $\gamma_t(T) = \lfloor \frac{n}{3}\rfloor$.
	}
	\label{fig:lowerbounds}
\end{figure}

Suppose that $j$ is an even number. If $\overline{M}'$ is neither $H_1$ nor $H_2$,  it contains a TDS  $D$ of size at most $f(n-4j)$ by the induction hypothesis. If $\overline{M}'$ is either $H_1$ or $H_2$, by Lemma~\ref{lem:ExceptionCases} (I) (ii) there exists a TDS $D$ in $\overline{M}'\cup \{v'_j\}$ of size 5 containing $v'_j$. Therefore, the set $D$ together with the 3-vertex sets $\{v_{i-1},u'_i,v_i\}$, for $i=2,4,\ldots ,j$, form a TDS of $T$  with size at most $f(n-4j)+3j/2$ in the first case and with size $5+3j/2$ in the second case. By Lemma~\ref{lem:fn}, $f(n-4j)+3j/2\le f(n)$ because $\frac{3j/2}{4j}< \frac{2}{5}$, and trivially $5+3j/2 \le \lfloor \frac{2}{5}(12+4j)\rfloor$ for even $j\ge 2$. Hence, $\gamma_t(T) \le f(n)$.

Suppose now that $j$ is an odd number. By Claim~\ref{Claim}, even if $\overline{M}'$ is either $H_1$ or $H_2$, we can obtain a set $D$ of vertices in $\overline{M}'$ such that the size of $D$ is at most $f(n-4j-1)+1$, $D$ contains the vertex $u'_1$ and $D$ dominates all vertices of $\overline{M}'$ except possibly $u'_1$. Since $M_1$ is a triangulated polygon, $u'_1$ and one of its neighbors, say $v_1$, form a TDS of $M_1$. Thus, by adding to $D$ the vertex $v_1$ and the 3-vertex sets $\{v_{i-1},u'_i,v_i\}$, for $i=3,5,\ldots ,j$, we obtain a TDS  of size at most $f(n-4j-1)+1+1+\frac{3}{2}(j-1)$. By Lemma~\ref{lem:fn}, $f(n-4j-1)+1+1+\frac{3}{2}(j-1)\le f(n)$ because $\frac {2+3(j-1)/2}{4j+1}\le 2/5$, hence $\gamma_t(T) \le f(n)$. Note that $v_j$ is dominated by $u'_j$.

Finally, assume that $T'$ has only one interior vertex $v$, so $T'$ is a wheel. By removing the vertices in $M_1, \ldots , M_{k-1}$ not in $P$, the vertices $u'_2, \ldots ,u'_{k-1}$ and the vertex $v$, we obtain a near-triangulation $\overline{M}'$ that coincides with $M_k$, and we argue as in the previous paragraphs depending on the parity of $j$. We remark that $\overline{M}'$ can be neither $H_1$ nor $H_2$, and that if $j$ is an odd number and $\overline{M}'=M_k$ is a triangulated pentagon  (so Claim~\ref{Claim} cannot be applied), then we chose a TDS of size 2 including $u'_1$ in $\overline{M}'$.
\end{proof}

\section{Final remarks}\label{sec:con}

In this paper, we proved that the total domination number for any $n$-vertex near-triangulation is at most $\lfloor \frac {2n}{5} \rfloor$ with two exceptions. The proof is by induction and is based on a new decomposition of some near-triangulations (the irreducible ones) into several near-triangulations, using what we call terminal polygons.

We believe that this new technique of partitioning near-triangulations will be useful to address some classical problems on triangulations from a different point of view, providing new insights on these problems. In particular, we think that some of the ideas given in the paper might be helpful to give new upper bounds in triangulations for other variants on the concept of domination.

To finish this paper, we give the following conjecture.

\begin{conj}
For any triangulation $T$ of order $n\ge 6$, $\gamma_t(T) \le \lfloor \frac{n}{3} \rfloor$.
\end{conj}

The conjecture is based on the following. The bound $\lfloor \frac {2n}{5} \rfloor$ on the total domination number in near-triangulations is tight, since there are near-triangulations achieving the bound. Figure~\ref{fig:tightMOP} shows one of these near-triangulations. However, all the examples reaching the bound that we know are MOPs. For triangulations, we feel that the total domination number should be smaller and close to $n/3$. This bound would be tight because there are triangulations reaching it. Figure~\ref{fig:tight} shows one of them. It consists of an octahedron containing in its interior other $k-1$ octahedra. The inter-octahedra region can be triangulated in any way. It is not difficult to see that any TDS for this triangulation must contain at least two vertices of each octahedron.

\section*{Acknowledgments}

A. Garc\'\i a, M. Mora and J. Tejel are supported by H2020-MSCA-RISE project 734922 - CONNECT; M. Claverol, A. Garc\'\i a, G. Hern\'andez, C. Hernando, M. Mora and J. Tejel are supported by project MTM2015-63791-R (MINECO/FEDER);  M. Claverol, C. Hernando, M. Mora and J. Tejel are supported by project PID2019-104129GB-I00 of the Spanish Ministry of Science and Innovation; M. Claverol is supported by project Gen. Cat. DGR 2017SGR1640; C. Hernando, M. Maureso and M. Mora are supported by project Gen. Cat. DGR 2017SGR1336; A. Garc\'\i a and J. Tejel are supported by project Gobierno de Arag\'on E41-17R (FEDER).

%%%%%%%%%%%%%%%%%%%%%%%%%

\bibliographystyle{abbrv}
\bibliography{bibliography}

\end{document}